\documentclass[11pt]{article}
\usepackage{amsmath,amssymb,amsfonts,amsthm} 
\usepackage{graphicx,subfigure,psfrag}
\usepackage{algorithm,algorithmic} 
\usepackage{url}
\usepackage{color,cancel}
\usepackage{tikz}
\usetikzlibrary{arrows} 
\usetikzlibrary{calc}
\usepackage{exscale,relsize}
\usepackage{hyperref}
\usepackage{mathdots}
\usepackage{footnote}

\newcommand{\algorithmicinput}{\textbf{Input:}}
\newcommand{\algorithmicoutput}{\textbf{Output:}}
\newcommand{\INPUT}{\item[\algorithmicinput]}
\newcommand{\OUTPUT}{\item[\algorithmicoutput]}

\newcommand{\R}{\mathbb{R}}
\newcommand{\C}{\mathbb{C}}
\newcommand{\Ss}{\mathbf{S}}
\newcommand{\Scal}{{\cal S}}

\def\bfS{\mathbf{S}}
\def\calJ{\mathcal{J}}
\def\bbR{\mathbb{R}}
\def\rrank{\mathop{\mathrm{rank}}}
\def\mvec{\mathop{\mathrm{vec}}}
\def\bmx{\begin{bmatrix}}
\def\emx{\end{bmatrix}}

\newcommand{\argmin}[1]{\underset{#1}{\operatorname{arg}\,\operatorname{min}}\;}
\renewcommand{\vec}[1]{\textnormal{vec}(#1)}
\newcommand{\rank}[1]{\textnormal{rank}(#1)}
\newcommand{\image}[1]{\textnormal{image}(#1)}

\newtheorem{theorem}{Theorem}
\newtheorem{lemma}{Lemma}
\newtheorem{remark}{Remark}
\newtheorem{conjecture}{Conjecture}
\newtheorem{proposition}{Proposition}

\title{\vspace*{-1cm}Factorization approach to structured low-rank approximation with applications\thanks{
The research leading to these results has received funding from the European Research Council under the European Union's Seventh Framework Programme (FP7/2007-2013) / ERC Grant agreement number 258581 ``Structured low-rank approximation: Theory, algorithms, and applications'', the Research Foundation Flanders (FWO-Vlaanderen), the Flemish Government (Methusalem Fund, METH1), and the Belgian Federal Government (Interuniversity Attraction Poles programme VII, Dynamical Systems, Control, and Optimization). MI is an FWO Pegasus Marie Curie Fellow.}}

\author{Mariya Ishteva\footnotemark[2]
\and Konstantin Usevich\footnotemark[2]
\and Ivan Markovsky\thanks{
Dept. ELEC,  Vrije Universiteit Brussel,
	Building K, Pleinlaan 2, B-1050 Brussels, Belgium
	({\tt \{mariya.ishteva, konstantin.usevich, ivan.markovsky\}@vub.ac.be}).} 
}
        
\date{}

\begin{document}

\maketitle

\begin{abstract}
We consider the problem of approximating an affinely structured matrix, for example a Hankel matrix, by a low-rank matrix with the same structure. This problem occurs in system identification, signal processing and computer algebra, among others. We impose the low-rank by modeling the approximation as a product of two factors with reduced dimension. The structure of the low-rank model is enforced by introducing a penalty term in the objective function. The proposed local optimization algorithm is able to solve the weighted structured low-rank approximation problem, as well as to deal with the cases of missing or fixed elements. In contrast to approaches based on kernel representations (in linear algebraic sense), the proposed algorithm is designed to address the case of small targeted rank. We compare it to existing approaches on numerical examples of system identification, approximate greatest common divisor problem, and symmetric tensor decomposition and demonstrate its consistently good performance.
\end{abstract}



\pagestyle{myheadings}
\thispagestyle{plain}
\markboth{M. ISHTEVA AND K. USEVICH AND I. MARKOVSKY}{STRUCTURED LOW-RANK APPROXIMATION BY FACTORIZATION}

\section{Introduction}
\label{sec:introduction}
Low-rank approximations are widely used in data mining, machine learning, and signal processing, as a tool for dimensionality reduction, feature extraction, and classification. In system identification, signal processing, and computer algebra, in addition to having low rank, the matrices are often structured, e.g., they have (block) Hankel, (block) Toeplitz, or (block) Sylvester structure. 
Motivated by this fact, in this paper, we consider the problem of approximating a given structured matrix $D$ by a matrix $\widehat{D}$ with the same structure and with a pre-specified reduced rank $r$.

\subsection{Problem formulation}
Formally, we consider the following problem
\begin{equation}
	\min_{\hat D} \|D - \hat D\|^2_{W}\,\quad\mbox{subject to }\quad 
	\rank{\hat D} \le r \mbox{ and } \hat D \mbox{ is structured,}
\label{slra_mat}
\end{equation}
where 
$D \in \mathbb{R}^{m \times n}$, $\widehat{D}\in \mathbb{R}^{m \times n}$, $r < \min(m,n)$ and 
$\|\cdot\|_{W}$ is a semi-norm on the space of matrices $\mathbb{R}^{m\times n}$, induced by a positive semidefinite matrix $W \in \mathbb{R}^{mn \times mn}$ as 
$$\|D\|^2_{W} := (\vec{D})^\top W \vec{D}.$$

Being able to deal with weights in \eqref{slra_mat} 
has a number of advantages in practice. 
First,  
due to sensor failure, malfunctioning of a communication channel, or simply due to unseen events, real-world data can have unknown (missing) elements. If repeating the experiments until all data are collected is not an option, for example because of high price of the experiments or high computational time, the missing data have to be approximated as well. The problem of estimating missing data is also known as the matrix completion problem and is well-studied in the case of unstructured matrices. In the case of structured matrices, however, this problem has few solutions \cite{slra-ext}. 
A natural way to deal with missing elements is to introduce zeros in the weight matrix 
at the positions corresponding to the missing elements. 

In addition, if prior knowledge is available about the importance or the correctness of each (noisy) element,
this knowledge can be encoded in 
the weight matrix.
Note also that finding the closest structured matrix to a given structured matrix with respect to the Frobenius norm can be encoded with $W$ being the identity matrix. 

The structures considered in \eqref{slra_mat} are affine structures. This class of structures includes many structures of interest and contains all linear structures. Moreover, it allows us to deal with fixed elements, i.e., to keep some elements of $D$ in the approximating matrix $\widehat{D}$. Unlike other approaches in the literature, we do not use infinite weights, but rather incorporate the fixed values in a special matrix.

\subsection{Solution approaches}
Existing algorithms to solve problem \eqref{slra_mat} can be classified into three groups:
i) based on local optimization, ii) using relaxations, or iii) using heuristics, such as the widely used Cadzow method \cite{CadWil90} or \cite{ZacSunEtal12}. Relaxation methods include subspace-based methods \cite{subspace-book,kung78} and, more recently, nuclear norm based methods \cite{LV09,LHV12,fazel2012hankel}. Local optimization algorithms use kernel or input/output (I/O) (also known as the structured total least squares (STLS) problem) representations of the rank constraint, as described in Table~\ref{tab:loc_opt_approaches}. 
Some of these algorithms cannot deal with fixed or missing elements or solve only special cases of problem \eqref{slra_mat}, e.g. the case of Frobenius norm.
\begin{table}[htb]%
\centering
\caption{Existing local optimization approaches for the structured low-rank approximation problem}
\begin{tabular}{l|l|l|l}
	Representation & Summary & Parameters & References\\[1mm]
	\hline\vspace*{1mm}
	Kernel & $R\widehat{D} = \mathbf{0}$ & \hspace*{-2mm} $R\in\R^{(m-r)\times m}$ & e.g., \cite{M07,lra-book,slra-efficient,borsdorf2012structured}\\[2mm]
	I/O& $\begin{bmatrix}X& I\end{bmatrix}\widehat{D} = \mathbf{0}$ & \hspace*{-2mm} $X\in\R^{(m-r)\times r}$ & e.g., \cite{slra-demoor,stls-block,park1999low,schuermans2004structured,mastronardi2000fast}\\[2mm]
	Image & $\widehat{D} = PL$ & \hspace*{-2mm} $P\in\R^{m\times r}\!$, $L\in\R^{r\times n}$&  \cite{chu2003structured}
\end{tabular}
\label{tab:loc_opt_approaches}
\end{table}

In this paper, we study an underrepresented point of view, namely the image representation of the rank constraint (also known as matrix factorization approach):
$$ \rank{\widehat{D}} \le r 
\iff \widehat{D} = PL\mbox{ for some }P\in \mathbb{R}^{m\times r},\, L\in \mathbb{R}^{r\times n}.$$
Although this view is widely represented in the literature on (unstructured) low-rank approximation, it is underrepresented in the case of structured low-rank approximation.
Imposing both low-rank and structure simultaneously with the image representation is a nontrivial problem \cite{chu2003structured}. 
The main difficulty comes from the fact that the structure has to be imposed on the approximation $\widehat D$ via the product of its factors $P$ and $L$.

\subsection{Our contribution}
We propose to resolve the problem of imposing the structure via the factors by using the penalty method \cite[\S 17.1]{NocWri06}. The structure can be imposed by introducing a penalty term in the cost function in \eqref{slra_mat}, representing the distance between the current iterate $PL$ and the space of structured matrices, i.e.,
\begin{equation*}
		\min_{P,\,L} \|D-PL\|^2_W +\lambda\,\mbox{dist}(PL, \mbox{its closest structured matrix}).
		\label{def:}
\end{equation*}
In this way the constrained optimization problem \eqref{slra_mat} becomes an unconstrained problem, which can be solved by an alternating projections (block coordinate descent) algorithm. 
To the best of our knowledge, this is the first detailed study of the matrix factorization view of structured low-rank approximation.
We apply the proposed algorithm on practically relevant and nontrivial simulation examples from system identification, computer algebra (finding a common divisor of polynomials with noisy coefficients), and symmetric tensor decomposition, and demonstrate its consistently good performance.
Large scale problems are not studied in the paper.

The main competitors of the proposed local optimization approach are the kernel-based algorithms, which aim at solving the same problem.
In contrast to kernel-based approaches which are meant for large rank $r$ (small rank reduction $m-r$), the proposed approach is more efficient for problems with small $r$. Moreover, for general affine structures, existing kernel approaches have restrictions on the possible values of the reduced rank $r$ \cite{stls-block}. With the new approach we can overcome this limitation.

Local optimization techniques need a good starting value in order to converge to an adequate local optimum. We solve a series of related subproblems with increasing penalty parameter $\lambda$. Due to the use of a small initial penalty parameter, the truncated SVD provides a good initial approximation for the initial subproblem. Each subsequent subproblem is initialized with the solution of the previous one, providing a good initialization for every subproblem.

Another known issue with the structured low-rank approximation problem is the possible non-existence of solution with fixed rank \cite{chu2003structured}. The proposed approach avoids this issue by requiring that the rank of the approximation is bounded from above by $r$. This way the feasible set is closed. Then if the fixed elements are all zeros (which is the case for most common structures), the feasible set is nonempty and solution always exists. We note, however, that in the non-generic case when 
the solution of \eqref{slra_mat} is of lower rank, the proposed algorithm has to be modified for the convergence results to hold.

Another advantage of the proposed algorithm is its simplicity. As we show in Section~\ref{sec:alg}, the proposed algorithm reduces to solving a sequence of least squares problems with closed form solutions. This makes it more robust to implementation issues, unlike the algorithm proposed in \cite{chu2003structured}.

Last but not least, we are able to solve the weighted structured low-rank approximation problem, as well as to deal with the cases of missing elements in the data matrix or fixed elements in the structure. These ``features'' have great impact on the applicability of the proposed approach, as demonstrated in the numerical section.

The rest of the paper is organized as follows. In Section~\ref{sec:structure}, we discuss the structure specification and how to obtain the closest structured matrix to a given unstructured matrix (orthogonal projection on the space of structured matrices). In Section~\ref{sec:optimization_problem}, 
we discuss the structure parameter point of view of the main optimization problem.
Our reformulations using penalty terms are proposed in Section~\ref{sec:regularized_formulation}. The main algorithm and its properties are discussed in Section~\ref{sec:algorithm}. In Section~\ref{sec:examples}, it is compared with existing approaches on numerical examples. In Section~\ref{sec:conclusions}, we draw our final conclusions.

\section{Structure specification and its use}
\label{sec:structure}
Commonly used structures include Hankel, block Hankel (system identification,  speech encoding, filter design), Toeplitz, block Toeplitz (signal processing and image
enhancement), Sylvester, extended Sylvester (computer algebra), and banded matrices with fixed bandwidth, see \cite[Table 1.1]{lra-book}, \cite[\S 2]{chu2003structured} and the references therein.
These matrices have a pattern for the position of their elements. For example, in a Hankel matrix $D\in\R^{m\times n}$, the elements along any anti-diagonal are the same, i.e.,
$$D= {\cal H}_m(p) = \begin{bmatrix}
	p_1 & p_2 & p_3 & \ldots & p_{n}\\
	p_2 & p_3 &  & \iddots & \\
	p_3 & \iddots &  & & \vdots\\
	\vdots & \iddots & \iddots &  & \\
	p_m & p_{m+1} & p_{m+2} &\ldots & p_{n_p}
\end{bmatrix}$$
for some vector $p\in\R^{n_p}$, $n_p = m+n-1$, called {\em structure parameter vector} for the matrix $D$.
Note that any $m\times n$ (unstructured) matrix can be considered as structured matrix with $n_p =mn$ structure parameters.

In this section, we first formally introduce the affine structures and then discuss the orthogonal projection on the space of structured matrices.

\subsection{Affine structures}
Formally, affine matrix structures are defined as
\begin{equation}
	\Scal(p) = S_0 + \sum_{k=1}^{n_p}S_kp_k,
\label{def:structure}
\end{equation}
where $S_0,S_1,\ldots, S_{n_p} \in\R^{m\times n}$, $p\in\R^{n_p}$ and 
$n_p\in\mathbb{N}$ is the number of structure parameters. We require $n_p$ to be minimal in the sense that
$$\image{\Scal}:=\{\Scal(p)\,|\, p\in\R^{n_p}\}$$
cannot be represented with less than $n_p$ parameters.
It is convenient to define the following matrix
\begin{equation}
	\Ss=\begin{bmatrix}\vec{S_1} & \,\cdots\, & \vec{S_{n_p}}\end{bmatrix}\in\R^{mn\times n_p},
\label{eq:bfs}
\end{equation}
where $\vec{X}$ denotes the vectorized matrix $X$.
The minimality of ${n_p}$ is equivalent to $\Ss$ having full column rank.
For simplicity, we assume that 
\begin{itemize}
	\item[{\bf (A)}] the matrix $\Ss$ consists of only zeros and ones and that there is at most one nonzero element in each row of the matrix $\begin{bmatrix}\vec{S_0} & \,\Ss\end{bmatrix}$, 
i.e., every element of the structured matrix corresponds to (at most) one element of $p$ or is a fixed element.
\end{itemize}
This assumption is satisfied for the common structures mentioned earlier ((block) Hankel, (block) Toeplitz, etc.)
and
implies that $n_p\leq mn$ and $\Ss^\top\,\vec{S_0} =\mathbf{0}.$

The matrix $S_0$ is introduced to handle fixed elements and is independent of the values of the structure parameters. The parameter vector $p$ is a vector of true parameters and does not include elements corresponding to fixed elements in the structure specification. In many cases (for example, Sylvester matrix in Section \ref{ex:gcd}), the fixed elements are zeros, and therefore $S_0 = \mathbf{0}$ (the structure is linear). However, we aim at dealing with the more general case of the arbitrary fixed elements.

\subsection{Orthogonal projection on image($\Scal$)}
\label{sec:orthogonal_projections}
Next we discuss the orthogonal projection of a matrix on to
the space of structured matrices $\image{\Scal}$. This projection is used in the optimization algorithm of Section~\ref{sec:algorithm}. 

\begin{lemma}
\label{lem:projection}
	For a structure $\Scal$ satisfying assumption {\bf (A)}, the orthogonal projection ${\cal P}_{\Scal}(X)$ of a matrix $X$ on $\image{\Scal}$ is given by 
	\begin{equation}
		{\cal P}_{\Scal}(X):=\Scal(\Ss^\dagger \,\textnormal{vec}\,(X)),\quad  where\quad \Ss^\dagger := (\Ss^\top\,\Ss)^{-1}\Ss^\top.
	\label{def_projop}
	\end{equation}
\label{thm:proj}
\end{lemma}
The proof is given in the appendix. With some modifications, this lemma also holds for any affine $\Scal.$
For future reference, using \eqref{def:structure}, \eqref{eq:bfs}, and \eqref{def_projop}, we also have the following equality
\begin{equation}
	\vec{{\cal P}_{\Scal}({X})} = \vec{S_0} + \Pi_\Ss\,\vec{X},
\label{eq:vec_proj}
\end{equation}
where $\Pi_\Ss = \Ss\, \Ss^\dagger = \Ss(\Ss^\top\,\Ss)^{-1}\Ss^\top$ is the orthogonal projector on the image of $\Ss.$

The effect of applying $\Ss^\dagger$ on a vectorized $m\times n$ matrix $X$ is producing a  structure parameter vector by averaging the elements of $X,$ corresponding to the same $S_k$.   
In particular, applying $\Ss^\dagger$ on a (vectorized) structured matrix extracts its structure parameter vector.
Further explanation is provided in the appendix.

\subsection{Parameter view of weighted structured low-rank approximation}
\label{sec:optimization_problem}
This section relates problem \eqref{slra_mat} to an approximation problem in the parameter norm
$$\|x\|_{\overline W}^2 := x^\top \overline W\, x,$$ 
where $\overline W\in\R^{n_p\times n_p}$ is a symmetric positive semidefinite matrix of weights.
If $\overline W$ is the identity matrix, then 
$\|\cdot\|_{\overline W}=\|\cdot\|_2$. 

\begin{lemma}
Problem \eqref{slra_mat} is equivalent to
the following problem
\begin{equation}
	\min_{\hat p} \|p - \hat p\|^2_{\overline W}\,\quad\mbox{subject to }\, \rank{\Scal(\hat p)}\leq r, 
\label{slra}
\end{equation}
with 
\begin{equation}
	\overline W = \Ss^\top W \Ss.
\label{WovlToW}
\end{equation}
\end{lemma}
\begin{proof}
Indeed,
\begin{equation*}
\begin{array}{rcl}
	\|\Scal(p)-\Scal(p_0)\|_W^2 
	& = & \vec{\Scal(p)-\Scal(p_0)}^\top W \, \vec{\Scal(p)-\Scal(p_0)} \\[1mm]
	& = & (\Ss (p-p_0))^\top W\, \Ss (p - p_0)\\[1mm]
	& = & (p-p_0)^\top \overline W\, (p-p_0) \\[1mm]
	& = & \|p-p_0\|_{\overline W}^2.
\end{array}
\label{rel:W}
\end{equation*}
Therefore, for $\overline W$ and $W$ related by \eqref{WovlToW}, problems \eqref{slra} and \eqref{slra_mat} are equivalent.\hfill
\end{proof}

In the literature, problem \eqref{slra} is sometimes referred to as the main problem from which \eqref{slra_mat} is derived.

\section{Penalized structured low-rank approximation}
\label{sec:regularized_formulation}
Each of the constraints in \eqref{slra_mat} can easily be handled separately.
Approximating a matrix by a structured matrix without imposing low-rank can be performed by orthogonally projecting the matrix on the space of structured matrices (see Section~\ref{sec:orthogonal_projections}). Unweighted low-rank approximation without imposing structure can be done using truncated singular value decomposition SVD~\cite{GVL96}.
However, imposing both low-rank and fixed structure on the approximation is nontrivial even in the unweighted case (when $\overline W =I_{n_p}$).
Likewise, due to the non-convexity of the rank constraint, the weighted low-rank approximation problem is difficult already in the unstructured case (when $\Ss = I$) \cite{SreJaa03}. 

We approach the weighted structured low-rank approximation problem from a new point of view, namely by a penalty technique. We propose two novel reformulations and discuss their relation to the original problem.

The main idea is to solve a series of related simpler subproblems, the solution of each subsequent problem being forced closer to the feasible region of the main problem. One of the requirements (low-rank or structure) will always be imposed, while the other one will be satisfied only upon convergence.
We have the following two choices (see Figure~\ref{fig:optimization_problems}):
\begin{figure}[htb]
\centering
\begin{tikzpicture}[>=triangle 45]
	\draw (0,0) node[minimum size = 1.5cm, draw, rounded  corners=8pt] (a11)
		{\begin{tabular}{cl}
			$\widehat D$ in \eqref{slra_mat}: & $\rightarrow\,$ low-rank constraint\\ 
				& $\rightarrow\,$ structure constraint
		\end{tabular}};
	\draw (-3.25,-2.5) node[minimum size = 1.5cm, draw, rounded  corners=8pt] (a1r) 
		{\begin{tabular}{l}
			$PL$ in \eqref{def:main_problem}:\\[1mm]
			$\rightarrow\,$ low-rank $\checkmark$\\ 
			$\rightarrow\,$ penalized structure deviation
		\end{tabular}};
	\draw (3.25,-2.5) node[minimum size = 1.5cm, draw, rounded  corners=8pt] (ar1)
		{\begin{tabular}{l}
			${\cal P}_{\Scal}(PL)$ in \eqref{def:alternative_main_problem}, \eqref{def:new}:\\[1mm]
			$\rightarrow\,$ penalized low-rank deviation\\
			$\rightarrow\,$ structure $\checkmark$
		\end{tabular}};
	\draw[->,thick] (a11.south) -- (a1r.north);
	\draw[->,thick] (a11.south) -- (ar1.north);
\end{tikzpicture}
\caption{Optimization problems}
\label{fig:optimization_problems}
\end{figure}
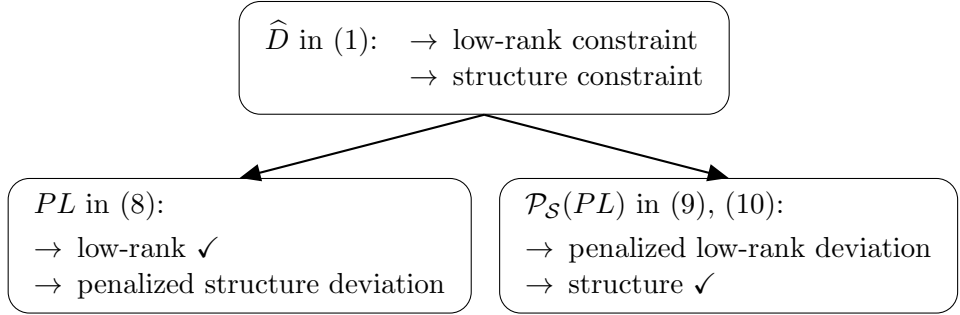
\begin{itemize}
	\item penalize the structure deviation
		\begin{equation}
		\min_{P,\,L} \|D-PL\|^2_W +\lambda\|PL - {\cal P}_{\Scal}({PL})\|^2_F,
		\label{def:main_problem}
		\end{equation}
		where $\lambda$ is a penalty parameter (discussed in Section~\ref{sec:lambda}),
		$\|\cdot\|_F$ stands for the Frobenius norm
		and ${\cal P}_{\Scal}({PL})$ is defined in \eqref{def_projop}, or
	\item penalize the low-rank deviation
		\begin{equation}
		\min_{P,\,L} \|D-{\cal P}_{\Scal}(PL)\|^2_W +\lambda\|PL - {\cal P}_{\Scal}({PL})\|^2_F.
		\label{def:alternative_main_problem}
		\end{equation}
\end{itemize}

The choice of $\lambda$ is discussed in Section~\ref{sec:lambda}. Note that for $\lambda = \infty$ the term $\|PL - {\cal P}_{\Scal}({PL})\|$ has to be zero and the three problems (\ref{slra_mat}), (\ref{def:main_problem}) and (\ref{def:alternative_main_problem}) are equivalent. The interpretations of (\ref{def:main_problem}) and (\ref{def:alternative_main_problem}) are however different. 
In \eqref{def:main_problem}, $PL$ is a matrix of low rank but it only approximately obeys the given structure. Forcing $\|PL - {\cal P}_{\Scal}({PL})\|^2_F$ to zero forces $PL$ to have the required structure as well. 
In \eqref{def:alternative_main_problem}, the iterate can be considered to be ${\cal P}_{\Scal}({PL}),$
i.e., at each iteration the iterate is a structured matrix but the low-rank constraint is only approximately satisfied. Forcing $\|PL - {\cal P}_{\Scal}({PL})\|^2_F$ to zero forces the iterate to have low-rank as well.

Since $\Ss$ is of full rank,
given $\overline W$ in problem \eqref{slra}, there are many possibilities to choose $W$ satisfying \eqref{WovlToW},
such that problems \eqref{slra_mat} and \eqref{slra} are equivalent. However, the following holds true.

\begin{remark}
	If $W$ satisfies \eqref{WovlToW}, problem (\ref{def:alternative_main_problem}) is independent of the choice of $W$
	and can be formulated using $\overline W$ 
	in the  following way
	\begin{equation}
		\boxed{\min_{P,\,L} \|p- \Ss^\dagger \vec{PL}\|^2_{\overline W} +\lambda\|PL - {\cal P}_{\Scal}({PL})\|^2_F.}
		\label{def:new}
	\end{equation}
\end{remark}

Because of this reason, we will focus on problem formulation \eqref{def:alternative_main_problem} and its equivalent representation \eqref{def:new}.
For a particular choice $W_* = (\Ss^\dagger)^\top {\overline W}\Ss^\dagger$ (which satisfies \eqref{WovlToW}), problem \eqref{def:main_problem} is equivalent to \eqref{def:new}.

\begin{remark}[Existence of solution]
A known issue with the structured low-rank approximation problem is the possible non-existence of solution with fixed rank \cite{chu2003structured}. The proposed approach avoids this issue by requiring that the rank of the approximation is bounded from above by $r$. This way the feasible set is closed. Then if $S_0=\mathbf{0}$ (which is the case for most common structures), then the feasible set is nonempty and solution always exists. We note, however, that in the non-generic case when 
the solution of \eqref{slra_mat} is of lower rank, the proposed algorithm has to be modified for the convergence results to hold.
\end{remark}

\section{The proposed algorithm}
\label{sec:algorithm}
In this section, we propose an algorithm in the framework of the penalty methods. We first discuss how the minimization problem~\eqref{def:new} can be solved for a fixed value of the penalty parameter $\lambda$ and then present the algorithmic and computational details related to the proposed algorithm.

\subsection{Description of the algorithm}
\label{sec:alg}
The main idea is to solve the minimization problem (\ref{def:new}) by 
alternatingly improving the approximations of $L$, for fixed $P$,
\begin{equation}
	\begin{array}{l}
		\displaystyle{\min_{L} \|p- \Ss^\dagger\, \vec{PL}\|^2_{\overline W} +\lambda\|PL - {\cal P}_{\Scal}({PL})\|^2_F}
	\end{array}
\label{def:subproblem1}
\end{equation}
and of $P$, for fixed $L$,
\begin{equation}
	\begin{array}{l}
		\displaystyle{\min_{P} \|p- \Ss^\dagger\, \vec{PL}\|^2_{\overline W} +\lambda\|PL - {\cal P}_{\Scal}({PL})\|^2_F.}
	\end{array}
\label{def:subproblem2}
\end{equation}

Let $I_n$ be the $n\times n$ identity matrix and let '$\otimes$' be the Kronecker product 
$$X\otimes Y = \begin{bmatrix}
 x_{11}Y & \cdots & x_{1n}Y\\
 \vdots & \ddots & \vdots\\
 x_{m1}Y & \cdots & x_{mn}Y
 \end{bmatrix},\quad
\mbox{for } X\in\R^{m\times n}.$$

\begin{lemma}
\label{lem:LS_reformulation}
Problems \textnormal{(\ref{def:subproblem1})} and \textnormal{(\ref{def:subproblem2})} are equivalent to the following least squares problems
\begin{equation}
\begin{array}{lcc}
(\ref{def:subproblem1}) & \Leftrightarrow &
	 \displaystyle{\min_{L}} \left\|
		\begin{bmatrix} \overline M\, \Ss^\dagger\\[2mm]\sqrt{\lambda}\Pi_{\Ss_\perp}\end{bmatrix}\,
		(I_n\otimes P)\,\vec{L}
		-\begin{bmatrix} \overline M p\\[2mm]\sqrt{\lambda}\vec{S_0}\end{bmatrix}\right\|^2_2,
\\ [10mm]
(\ref{def:subproblem2}) & \Leftrightarrow & 
	\displaystyle{\min_{P}} \left\|
		\begin{bmatrix} \overline M\, \Ss^\dagger\\[2mm]\sqrt{\lambda}\Pi_{\Ss_\perp}\end{bmatrix}\,
		(L^\top\otimes I_m)\,\vec{P}
		-\begin{bmatrix} \overline M p\\[2mm]\sqrt{\lambda}\vec{S_0}\end{bmatrix}\right\|^2_2,
\end{array}
\label{min2a}
\end{equation}
for an $\overline M\in\R^{n_p \times n_p}$ with $\overline W = \overline M^\top \overline M$
and $\Pi_{\Ss_\perp} = (I_{mn} - \Pi_{\Ss})$ being the orthogonal projector on the left kernel of $\Ss$.
\end{lemma}

The proof is given in the appendix.

Both reformulations in Lemma~\ref{lem:LS_reformulation} are least squares problems and can be solved in closed form.
For fixed $\lambda$, we propose an algorithm in the framework of alternating least squares and block coordinate descent, namely we alternatingly improve the approximations of $L$ and of $P$ by solving the least squares problems in (\ref{min2a}). We discuss the update strategy for $\lambda$ in Section~\ref{sec:lambda}.
The choice of initial approximation $P_0$ for $P$ and the stopping criteria are discussed in Section~\ref{sec:init_stop}. 
The summary of the proposed algorithm is presented in Algorithm~\ref{alg:reg_slra}.

\begin{algorithm}[htb]
\caption{Structured low-rank approximation by factorization}
\label{alg:reg_slra}
\begin{algorithmic}[1]
\INPUT $p\in\R^{n_p}$, $S_0\in\R^{m\times n},$ $\Ss\in\R^{mn\times n_p}$,
	$\overline W = \overline M^\top \overline M \in \R^{n_p\times n_p}$, $r\in\mathbb{N}$,
	$P_0\in\R^{m\times r}\!$.
\OUTPUT Factors $P\in\R^{m\times r}$ and $L\in\R^{r\times n},$ corresponding to a 
	structured low-rank approximation problem (\ref{def:new}).
\STATE Set $\Ss^\dagger = (\Ss^\top  \Ss)^{-1}\Ss^\top.$
\STATE Set $P=P_0$, $\lambda_1 =1.$  
\FOR {$j = 1,2,\ldots$ until a stopping criterion is satisfied}
	\FOR {$k = 1,2,\ldots$ until a stopping criterion is satisfied}
		\STATE Update $L$ from \eqref{def:subproblem1}.
		\STATE Update $P$ from \eqref{def:subproblem2}.
	\ENDFOR
	\STATE Set $\lambda_{j+1}$ such that $\lambda_{j+1}>\lambda_{j}.$
\ENDFOR
\end{algorithmic}
\end{algorithm}

Due to the simplicity of Algorithm~\ref{alg:reg_slra}, dealing with weights, missing elements and structures with fixed elements is straightforward. The weight matrix $\overline W$ and the matrix with fixed elements $S_0$ are readily introduced in \eqref{min2a} and thus also in Algorithm~\ref{alg:reg_slra}.
Dealing with missing elements is realized by introducing zeros in the weight matrix $\overline W$ at the positions corresponding to the missing elements. 
Numerical examples are introduced in Section~\ref{sec:examples}.

\subsubsection{Parameter $\lambda$}
\label{sec:lambda}
In theory, if we fix $\lambda = \infty$, then (\ref{def:new}) is the
exact structured low-rank approximation problem. In practice, we may fix $\lambda$
to a ``large enough'' value and the solution $PL$ is only approximately a structured matrix.
The higher the value of $\lambda$, the better the structure constraint is satisfied; 
however, too large values may lead to numerical issues. 

Alternatively, adaptive schemes for updating $\lambda$ are also possible.
We can start from a small value and increase it with each iteration or set of iterations. This way we allow 
the algorithm
to move to a ``good region'' first and then impose more strictly the structure constraint \cite{NocWri06}. 
The following strategy has been proposed in \cite[\S 17.1]{NocWri06}: if solving the previous subproblem was expensive, increase $\lambda$ only modestly, e.g., $\lambda_{j+1}=1.5\lambda_j$. If solving the previous subproblem was cheap, increase $\lambda$ more ambitiously, $\lambda_{j+1}=10\lambda_j.$

\subsubsection{Initial guess and stopping criterion}
\label{sec:init_stop}
Let $D=U_r \Sigma_r V_r^\top$ be the truncated SVD of the given matrix $D$ ($D=\Scal(p)$). 
We initialize $P$ by $U_r$.
For small $\lambda$, the second term in the objective function is also small. In particular, for $\lambda = 0$, problem \eqref{def:main_problem} is exactly the (unstructured) low-rank approximation problem, for which truncated SVD provides a globally optimal solution (in the unweighted case). Problem \eqref{def:alternative_main_problem} and thus Problem \eqref{def:new} are equivalent to Problem \eqref{def:main_problem} for $\lambda = \infty$ and are still closely related for small $\lambda$ as well. This is why the truncated SVD provides a good initial approximation for Problem \eqref{def:alternative_main_problem} and thus also for Problem \eqref{def:new}. We solve a series of related subproblems with increasing penalty parameter $\lambda$. Each subsequent subproblem is initialized with the solution of the previous one, providing a good initialization for every subproblem.

Consider the following stopping criteria. For 
$\lambda$ fixed to $\lambda_j$, stop when the 
derivatives of the objective function are smaller than $\tau_j$ with $\tau_j \rightarrow 0$ as $j\rightarrow \infty.$ 
In practice, if we do not want to compute derivatives, the algorithm is often 
stopped when there is little change in the column space of $P$, 
although further progress can potentially be achieved.
Note that for small $\lambda$ we do not need to solve the problem exactly.
Thus, we can stop 
earlier and avoid slow convergence. Only  
when $\lambda$ becomes large, good approximation is required.
We stop iterating when $\lambda$ is ``large enough'', e.g, $10^{14}$. 

We declare that $PL$ is a structured matrix if
$$\frac{\|PL - {\cal P}_{\Scal}({PL})\|^2_F}{\|PL\|^2_F}<\varepsilon,$$
for a small $\varepsilon$, e.g., $\varepsilon = 10^{-12}.$

\subsubsection{Computational complexity}
The main computational cost is due to solving the least squares problems in \eqref{min2a}, which is equivalent to solving two systems of linear equations. The computational cost for constructing the matrices and the vector of the systems is smaller than 
the cost for solving the system. The size of the systems' matrices are $(n_p + mn)\times rn$ and $(n_p + mn) \times rm,$ respectively. Suppose that $m\leq n$ and recall that $n_p\leq mn.$ Then, the cost for one 
step of the proposed algorithm is 
$$O((rn)^2 (n_p + mn)) = O(n^3mr^2).$$
Note that this estimate does not take into account the available structure and sparsity of the matrices in \eqref{min2a}. If the structure and sparsity are exploited, faster computation can be performed. 
Additionally, compared to the kernel approaches, which are fast for large values of $r$ (preferably $r=m-1$), the proposed approach is designed for small values of $r$.
If the structure is not taken into account, the kernel approaches also have computational cost that is cubic in $n$, namely $O((m-r)^3 n^3)$,
and moreover their cost per iteration is higher in $m$ for small~$r$.

\subsection{Convergence properties}
Algorithm~\ref{alg:reg_slra} falls into the setting of the pe\-nal\-ty methods for constrained optimization~\cite[\S 17.1]{NocWri06} whose convergence properties are well understood.
{For fixed $\lambda$}, the proposed algorithm is an alternating least squares (or block coordinate descent) algorithm. 
Since we can solve the least squares problems in (\ref{min2a}) in closed form, every limit point of the generated sequence is a stationary point~\cite[\S 1.8.1, \S 2.7]{bertsekas1999np}\cite{grippo2000cbn}. 
The convergence rate of these methods is linear.
	
For the convergence properties of the algorithm as $\lambda\rightarrow \infty$, we have the following theorem borrowed from the theory of the quadratic penalty method.
We first define $c(P,L) = \begin{bmatrix} c_1(P,L) & \cdots & c_{mn}(P,L)\end{bmatrix}^\top$ as the vector of penalties
\begin{equation}
			c(P,L) := \mvec (PL - {\cal P}_{S}(PL)).
\label{eq:c}
\end{equation}
\begin{theorem}\textnormal{\cite[Th. 17.2.]{NocWri06}}
\label{th:convergence}
For $\tau_j\rightarrow 0$ and $\lambda_j \rightarrow \infty,$ if a limit point $(P,L)^\ast$ of the sequence $\{(P,L)_j\}$ generated by Algorithm~\ref{alg:reg_slra} is infeasible, it is a stationary point of the function $\|c(P,L)\|_2^2.$ On the other hand, if a limit point $(P,L)^\ast$ is feasible and the constraint gradients $\nabla c_i((P,L)^\ast)$ are linearly independent, then $(P,L)^\ast$ is a KKT (Karush-Kuhn-Tucker) point for problem \eqref{slra_mat}. For such 
points, we have for any infinite subsequence ${\cal K}$ such that 
$\lim_{j\in{\cal K}} (P,L)_j = (P,L)^\ast$ that
	$$\lim_{j\in{\cal K}} \,-\lambda_j c_i((P,L)_j) = \nu^\ast_i, \quad\mbox{for all } i=1,\ldots,mn,$$
where $\nu^\ast$ is the multiplier vector that satisfies the KKT conditions.
\end{theorem}

We next discuss the applicability of the theorem in our setting and some important special cases. We start with a discussion on the constraints.

\subsubsection{Constraints}
There are $mn$ constraints in total, defined by
\begin{equation}
\begin{array}{lcl}
	PL = {\cal P}_{\Scal}(PL) \,\,\Longleftrightarrow\,\, c(P,L) = \mathbf{0} 
	& \Longleftrightarrow & \Pi_{\Ss_\perp}\,\vec{PL} -\vec{S_0} = \mathbf{0}\\[1mm]
\end{array}
\label{eq:constraint}
\end{equation}
Note, however, that by assumption {\bf (A)}, $\Ss^\top \vec{S_0}=0$ and thus
$$\Pi_{\Ss_\perp} \vec{S_0} = (I_{mn} - \Pi_\Ss)\vec{S_0} = (I_{mn} - \Ss(\Ss^\top\Ss)^{-1}\Ss^\top)\vec{S_0} = \vec{S_0}.$$
Thus, \eqref{eq:c} can be written as
\begin{equation}
\begin{array}{lcl}
	c(P,L) & = & \Pi_{\Ss_\perp}\,(\vec{PL} -\vec{S_0}) \\[1mm]
	& = & \Pi_{\Ss_\perp}\,((L^\top \otimes I_m)\,\vec{P} -\vec{S_0})\\[1mm]
	& = & \Pi_{\Ss_\perp}\,((I_n\otimes P)\,\vec{L} -\vec{S_0}).
\end{array}
\label{eq:constraint2}
\end{equation}

The optimization variables are the entries of $P$ and the entries of $L$ and
the above equivalences show that each constraint is linear in each optimization variable.
It can be concluded that the matrix of constraint gradients
(the Jacobian of the vector of constraints) is
\begin{equation}
	\calJ_c(P,L) = \Pi_{\Ss_\perp}\,\begin{bmatrix}L^\top \otimes I_m & I_n\otimes P\end{bmatrix} 
	\in \bbR^{mn \times (mr+nr)},
\label{eq:constraint_gradients}
\end{equation}
where the gradient of the $i$-th constraint is the $i$-th row of the matrix.

\subsubsection{$S_0=\mathbf{0}$ and feasibility of the limit points}
The class of structures with $S_0 = \mathbf{0}$ is particularly interesting because of the following reasons:
\begin{itemize}
	\item Most common structures ((block-, mosaic-, quasi-) Hankel, (extended) Sylvester, etc.) are in this class.
	\item Problem \eqref{slra_mat} always has a solution since the feasible set is closed and nonempty (the zero matrix is in the feasible set).
	\item The limit points of the sequence generated by Algorithm~\ref{alg:reg_slra} are feasible, as follows directly from the following proposition.
\end{itemize}
\begin{proposition}
\label{prop:feasibility}
Let $S_0 = \mathbf{0}$.
Then any stationary point of $\|c(P,L)\|^2_2$ is feasible.
\end{proposition}

The proof in given in the appendix.

\subsubsection{Equivalent set of constraints}
Theorem~\ref{th:convergence} cannot be applied directly to the set of constraints $c(P,L),$
because $\rank{\calJ_{c}(P,L)} \leq \rank{\Pi_{\Ss_\perp}} = mn-n_p < mn.$
We will next transform the set of constraints $c(P,L) = \mathbf{0}$  
into an equivalent set of constraints $\tilde c(P,L)= \mathbf{0},$ such that 
the conditions of Theorem~\ref{th:convergence} can be satisfied for  $\tilde c(P,L)$.

Let $\Ss_\perp\in\R^{mn\times (mn-n_p)}$ be a matrix, whose columns span the orthogonal complement of $\Ss$ with 
	$\Ss_\perp^\top \Ss_\perp = I_{mn-n_p}.$
	Then we also have 
		$${\Pi_{\Ss_\perp} = \Ss_\perp (\Ss_\perp^\top \Ss_\perp)^{-1} \Ss_\perp^\top =
			\Ss_\perp\Ss_\perp^\top}.$$
	Let 
		$$\tilde c(P,L) = \Ss_\perp^\top(\vec{PL} - \vec{S_0}).$$
	Then
	\begin{equation}
		\begin{array}{lcl}
		\|c(P,L)\|^2_2 & = & \|\Pi_{\Ss_\perp}(\vec{PL} - \vec{S_0})\|^2_2 \\[1mm]
		 & = & \|\Ss_\perp\Ss_\perp^\top(\vec{PL} - \vec{S_0})\|^2_2 \\[1mm]
		 & = & \|\Ss_\perp^\top(\vec{PL} - \vec{S_0})\|^2_2\\[1mm]
		 & = & \|\tilde c(P,L)\|^2_2.
		\end{array}
	\label{eq:projected_c}
	\end{equation}
	Note that from \eqref{eq:projected_c}, it follows that
		$$c(P,L) = \mathbf{0}\quad \iff \quad \tilde c(P,L) = \mathbf{0}.$$
Similarly to \eqref{eq:constraint_gradients}, the Jacobian of the vector of constraints $\tilde c(P,L)$ is 
\begin{equation}
	\calJ_{\tilde c}(P,L) = \Ss_\perp^\top \bmx L^{\top} \otimes I_m  & I_n \otimes P \emx\in
	\R^{(mn-n_p)\times (mr+nr)}.
\label{eq:new_constraint_gradients}
\end{equation}

\subsubsection{Independence of the constraint gradients}
We will need the following lemma.
\begin{lemma}
\label{lem:rank_2nd_matrix}
Let $P\in\R^{m\times r}$ and $L\in\R^{r\times n}$, then\\[-2mm]
\begin{enumerate}
	\item $\rrank (\begin{bmatrix}L^\top \otimes I_m & I_n\otimes P\end{bmatrix}) \leq  mr+nr-r^2.$\\[-2mm]
	\item The equality holds if and only if $\rank{P} = \rank{L} = r.$
\end{enumerate}
\end{lemma}
The proof in given in the appendix.

\begin{remark}[The condition $\rank{P}=\rank{L}=r$]
The condition that $P$ and $L$ have rank $r$ is generically satisfied. Nongeneric cases would appear when the exact rank-$r$ problem does not have a solution, i.e., when the solution of \eqref{slra_mat} is of rank lower than $r$. In this case, Algorithm~\ref{alg:reg_slra} can be modified to detect rank deficiency in $P$ and in $L$ and reduce their corresponding dimensions, so that the reduced matrices have full column- and full row-rank, respectively.
\end{remark}

From \eqref{eq:new_constraint_gradients} and Lemma~\ref{lem:rank_2nd_matrix} we have that
\begin{equation}
\rrank (\calJ_{\widetilde c}(P,L)) \leq  \min(mn-n_p, mr+nr-r^2)
\label{eq:min2}
\end{equation}
and thus the following remark holds true.

\begin{remark}[Necessary and sufficient condition for independence of the constraint gradients]
\label{rem:iff_condition_independence}
An assumption of Theorem \ref{th:convergence} is that the constraint gradients are linearly independent in the feasible set,
i.e., that $\calJ_{\tilde c}(P,L)$ has full row rank and thus
\begin{equation}
\rank{\calJ_{\tilde c}(P,L)} = mn-n_p.
\label{eq:rank_J_cond}
\end{equation}
		From \eqref{eq:min2}
		it follows that 
		a necessary condition for \eqref{eq:rank_J_cond} to hold
		is
		\begin{equation}
			mn-n_p \leq mr+nr-r^2.
		\label{necessary_cond}
		\end{equation}
		This condition is often satisfied in practice, for example, for all Hankel approximations 
		with rank reduction by $1$.

		Generically, the condition \eqref{necessary_cond} is also a 
		sufficient condition for \eqref{eq:rank_J_cond} to hold.
\end{remark}
Apart from the case described in Remark~\ref{rem:iff_condition_independence}, another extreme case
in \eqref{eq:min2} may happen, namely that
$\rrank (\mathcal{J}_{\widetilde{c}} (P,L)) = mr+nr-r^2$.
\begin{lemma}\label{prop:prop2}
If a limit point $(P_*,L_*)$ of the sequence $\{(P,L)_j\}$ generated by Algorithm~\ref{alg:reg_slra}
is feasible, and $\rrank (\mathcal{J}_{\widetilde{c}} (P_*,L_*)) = mr+nr-r^2$, then 
\begin{enumerate}
\item there exists an affine subspace $\mathcal{L} \subset \mathbb{R}^{mn-n_p}$ of dimension $(m-r)(n-r) - n_p$ such that any $\nu \in \mathcal{L}$ is a multiplier that satisfies the KKT conditions.
\item $\rrank (\mathcal{J}_{\widetilde{c}} (P,L)) = mr+nr-r^2$ in a neighborhood of $(P_*,L_*)$
\end{enumerate}
\end{lemma}
The proof in given in the appendix.

\begin{proposition}
If a limit point $(P_*,L_*)$ of the sequence $\{(P,L)_j\}$ generated by Algorithm~\ref{alg:reg_slra}
is feasible, and $\rrank \mathcal{J}_{\widetilde{c}} (P_*,L_*) = mr+nr-r^2$,
then the limit point satisfies first order necessary conditions for \eqref{slra_mat}.
\end{proposition}
\begin{proof}
Lemma~\ref{prop:prop2} implies that if $\rrank \mathcal{J}_{\widetilde{c}} (P_*,L_*) = mr+nr-r^2$, the limit point $(P_*,L_*)$ satisfies the so-called Constant Rank Constraint Qualification \cite{Andreani.etal10JoOTaA-Constant}. In this case the KKT conditions are necessary for the point to be a local minimum of the constrained problem, see \cite{Andreani.etal10JoOTaA-Constant}.\hfill
\end{proof}

Finally, it may happen that $\rrank (\mathcal{J}_{\widetilde{c}} (P,L)) < \min(nm-n_p, mr+nr-r^2).$
In this case, we conjecture the following.
\begin{conjecture}
A feasible limit point satisfies the first order necessary conditions if
the rank of Jacobian is constant in the neighborhood of that point.
\end{conjecture}

\section{Numerical experiments}
\label{sec:examples}

In this section we apply the proposed algorithm on three different problems, namely, system identification, finding a common divisor of polynomials (with noisy coefficients), and symmetric tensor decomposition. Although these problems arise in completely different fields, are essentially different, and require different features (weights, fixed elements, or missing elements), we demonstrate the consistently good performance of Algorithm~\ref{alg:reg_slra}.

\subsection{Related algorithms}
In our {\sc Matlab} simulations, we compare Algorithm~\ref{alg:reg_slra} with Cadzow's algorithm~\cite{CadWil90} and the kernel-based algorithm {\tt slra} \cite{slra-software}, which also aim to solve problem (\ref{slra_mat}).
The former is popular in signal processing applications due to its simplicity and the latter has been recently extended to work with missing data.
We next briefly summarize the main ideas behind these algorithms.

Cadzow's algorithm \cite{CadWil90} consists of repeating the following two main steps
\begin{itemize}
	\item ``project'' the current structured approximation to a low-rank matrix, e.g., using truncated SVD,
	\item project the current low-rank matrix to the space of structured matrices.
\end{itemize}
As shown in \cite{DMo94}, Cadzow's algorithm 
converges to a structured matrix of rank $r$, but not necessarily to a stationary point of the optimization problem.

The {\tt slra} algorithm~\cite{slra-efficient}
is based on the kernel representation of the rank constraint, i.e.,
$$\rank{\widehat{D}} \le r \iff R\widehat{D} = \mathbf{0}\mbox{ for some full row rank matrix }R\in \mathbb{R}^{(m-r)\times m}$$
and uses the variable projection method \cite{varpro}, i.e., reformulation of the problem as inner and outer optimization, where the inner minimization admits an analytic solution. The outer problem is a nonlinear least squares problem and is solved by standard local optimization methods, e.g., the Levenberg--Marquardt method \cite{M63}. Generalization to problems with fixed and missing data is presented in \cite{slra-ext}. 
For general affine structures, {\tt slra} has a restriction on the possible values of the reduced rank $r$. Algorithm~\ref{alg:reg_slra} overcomes this limitation.

\subsection{Autonomous system identification}
\label{ex:sysid}

In this section, we will
use the fact that the proposed algorithm can handle {\it weighted} norms and {\it missing} elements. The considered matrices have Hankel structure.

\subsubsection*{Background}
In system theory, a discrete-time autonomous linear time-invariant dynamical system~\cite{lra-book} can be defined by a difference equation
\begin{equation}
	{\theta_0}\, y(t) + {\theta_1}\, y(t+1) + \cdots + {\theta_\ell}\, y(t+\ell) = 0, 
	\quad \mbox{for } t = 1,\ldots, T-\ell,
\label{eq:diff}
\end{equation}
where $\ell$ is the order of the system
and $\theta=\begin{bmatrix}\theta_0  & \theta_1 & \cdots  & \theta_\ell \end{bmatrix}\in\R^{\ell+1}$ is a non-zero vector of model parameters.
The problem of system identification is: estimate $\theta$, 
given a response $y = [y(1),\ldots, y(T)] \in \R^{T}$ of the system.

The difference equation \eqref{eq:diff} can equivalently be represented as
\begin{equation}
	{\begin{bmatrix}\theta_0  & \theta_1 & \cdots  & \theta_\ell \end{bmatrix}}
	\underbrace{\begin{bmatrix}
	y(1) & y(2) & y(3) & \ldots & y(T-\ell)\\
	y(2) & y(3) & y(4) & \iddots & \\
	y(3) & \iddots &  & & \vdots\\
	\vdots & \iddots & \iddots &  & \\
	y(\ell+1) & y(\ell+2) & y(\ell+3) &\ldots & y(T)
	\end{bmatrix}}_{{\cal H}_{\ell+1}(y)}=\mathbf{0}.
\label{eq:sysid_matrixeq}
\end{equation}
It follows from \eqref{eq:sysid_matrixeq} that the Hankel matrix ${\cal H}_{\ell+1}(y)$ is rank deficient, i.e., 
$$\rank{{\cal H}_{\ell+1}(y)}\leq \ell.$$ 

In the more realistic noisy case however, \eqref{eq:diff} and \eqref{eq:sysid_matrixeq} hold only approximately. 
The problem of identifying the system then reduces to finding a rank-$\ell$ approximation of ${\cal H}_{\ell+1}(y).$
The parameter $\theta$ can then be computed from the null space of the obtained approximation.

If enough samples are provided, which is usually the case in engineering applications, another possible reformulation of \eqref{eq:diff} is the following
\begin{equation}
	{\begin{bmatrix}\theta_0  & \theta_1 & \cdots  & \theta_\ell & 0\\[2mm]
		0 & \theta_0  & \theta_1 & \cdots  & \theta_\ell \end{bmatrix}}
	\underbrace{\begin{bmatrix}
	y(1) & y(2) & y(3) & \ldots & y(T-\ell-1)\\
	y(2) & y(3) & y(4) & \iddots & \\
	y(3) & \iddots &  & & \vdots\\
	\vdots & \iddots & \iddots &  & \\
	y(\ell+2) & y(\ell+3) & y(\ell+4) &\ldots & y(T)
	\end{bmatrix}}_{{\cal H}_{\ell+2}(y)}=\mathbf{0}.
\label{eq:sysid_matrixeq2}
\end{equation}
Compared to \eqref{eq:sysid_matrixeq}, the matrix ${\cal H}_{\ell+2}(y)$ in \eqref{eq:sysid_matrixeq2} has more rows and less columns but its rank is still at most $\ell$. 
The rank deficiency of ${\cal H}_{\ell+2}(y)$ is however at least $2$ since ${\cal H}_{\ell+2}(y)$ has $\ell+2$ rows. 
(This can also be concluded from the fact that there are now $2$ linearly independent vectors in the null space of the matrix.)
We can continue reshaping by adding rows and removing columns, e.g., until we get a square matrix ${\cal H}_{(T+1)/2}(y)$ if $T$ is odd or an almost square matrix ${\cal H}_{T/2}(y)$ if $T$ is even. This could be useful since there are indications that truncated SVD of an (almost) square Hankel matrix relates to a better (initial) noise reduction. Note that for large $T$ ($T\gg 2\ell$), the matrix ${\cal H}_{(T+1)/2}(y)$ (${\cal H}_{T/2}(y)$) would have low rank compared to its dimensions ($\ell \ll T/2$), which is the case the proposed algorithm aims for.

\subsubsection*{Example: system identification in Frobenius norm}
In this example, we will use the fact that the proposed algorithm can work with {\it weighted} norms. In particular, in order to approximate a structured matrix $\Scal(p)$ with a low-rank structured matrix $\Scal(\hat p)$ in Frobenius norm, i.e.,
$$\min_{\hat p} \|\Scal(p) - \Scal(\hat p)\|^2_F \,\quad\mbox{subject to }\quad \rank{\Scal(\hat p)} \le r$$
we need to take $\overline W$ from \eqref{slra} and \eqref{def:new} to be a diagonal matrix with weights equal to the number of occurrences of each structure parameter $p_i,\,i=1,\ldots, n_p.$ 
We consider the Frobenius norm as a distance measure between the data matrix and the approximation to facilitate the comparison with Cadzow's algorithm. 
In the next example, we illustrate the performance of the proposed algorithm with respect to the $2$-norm $\|p-\hat p\|^2_2$.

The considered true (noiseless) signal $y_0$ is the sum of the following two exponentially modulated cosines
\begin{equation}
	\begin{array}{rcl}
		y_0(t) & = & y_{0,1}(t) + y_{0,2}(t),\\[1.5mm]
		y_{0,1}(t) & = & \quad 0.9^t\, \cos(\frac{\pi}{5}\,t),\\[1.5mm]
		y_{0,2}(t) & = & \frac{1}{5}\, 1.05^t \cos(\frac{\pi}{12}\,t + \frac{\pi}{4}),
	\end{array}
\label{eq:sysid_y0}
\end{equation}
$t=1,\ldots, 50$, shown in Figure~\ref{fig:sysid_components}. 
\begin{figure}[htb]%
\centering
	\psfrag{cos 1}{$y_{0,1}$}
	\psfrag{cos 2}{$y_{0,2}$}
	\psfrag{x}{$t$}
	\psfrag{y}{\hspace*{-1.1cm}$y_{0,1}(t),\, y_{0,2}(t)$}
	\includegraphics[width=0.7\columnwidth]{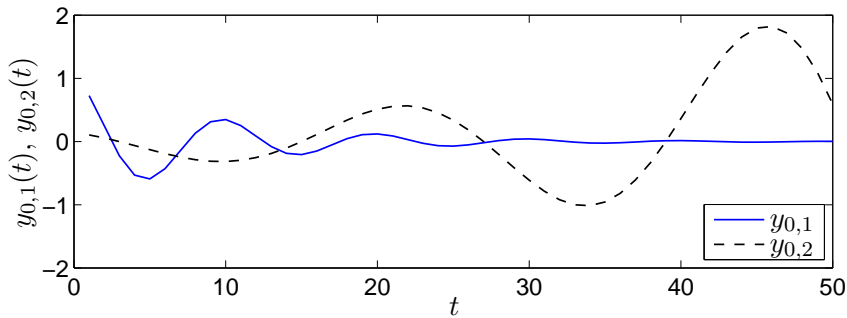}%
\caption{True components in the system identification examples.}%
\label{fig:sysid_components}%
\end{figure}
The rank of the corresponding Hankel matrix is $4$, i.e., $\rank{{\cal H}_m(y_0)} =4$, for $m=5,6,\ldots, 46$. We added noise in the following way
\begin{equation}
	y(t) = y_0(t) + 0.2\, \frac{e(t)}{\|e\|_2}\, \|y_0(t)\|_2,
\label{eq:sysid_y}
\end{equation}
where $e(t)$ were drawn independently from the normal distribution with zero mean and unit standard deviation. The added noise increases the rank of ${\cal H}_m(y_0)$, so rank-$4$ approximation has to be computed. We denote the approximations by $\hat y$.

We ran Algorithm~\ref{alg:reg_slra}, {\tt slra} and Cadzow's algorithm. 
To facilitate the comparison with {\tt slra}, we first set $m=r+1$ (i.e., $m=5$).
The initial approximation was obtained using the truncated SVD. 
Since {\tt slra} is sensitive to the initial approximation, in addition to its default initialization, we ran {\tt slra} starting from the solution obtained by Kung's method \cite{kung78} (row ``Kung $\rightarrow$ {\tt slra}'' in Table~\ref{sysid_Fro_comparison_rectangular}), which is a heuristic method for Hankel structured low-rank approximation.

After $1000$ iterations, Cadzow's algorithm still did not converge and the rank of its structured approximation was $5$ instead of $4$, with smallest singular value of the approximation $0.0027$. 
The solution of Algorithm~\ref{alg:reg_slra} had the smallest approximation error, see Table~\ref{sysid_Fro_comparison_rectangular}.
\begin{table}[htb]%
\centering
\caption{Numerical errors of the initial approximation (by SVD), Cadzow's algorithm, {\tt slra}, Kung's heuristic algorithm, {\tt slra} initialized with Kung's algorithm's solution, and the proposed algorithm (Algorithm~\ref{alg:reg_slra}), for the example \eqref{eq:sysid_y0}--\eqref{eq:sysid_y} with $m=5$.}
\begin{tabular}{l|r|r|l}
	 & $\|\Scal(y) - \Scal(\hat y)\|^2_F$ & $\|\Scal(y_0) - \Scal(\hat y)\|^2_F$ & Remarks\\
	\hline
	{\tt init.approx.} & $37.7107$ & $36.1934$ & \\
	\hline
	Cadzow & $(4.4916)$ & $(1.0740)$ & {\bf incorrect rank}\\
	\hline
	{\tt slra} & $11.3892$ & $7.0164$ & \\
	\hline
	{\tt Kung}~\cite{kung78} & $ 1.6526\cdot10^{13}$ & $1.6526\cdot 10^{13}$ & heuristic\\
	\hline
	{\tt Kung} $\rightarrow$ {slra} & ${11.3892}$ & $7.0172$ & \\
	\hline
	Algorithm~\ref{alg:reg_slra} & $\mathbf{4.6112}$ & $ \mathbf{0.4848}$ & $\frac{\|PL - {\cal P}_{\Scal}({PL})\|^2_F}{\|PL\|^2_F} < 10^{-24}$
\end{tabular}
\label{sysid_Fro_comparison_rectangular}
\end{table}
The error of the initial approximation reported in Table~\ref{sysid_Fro_comparison_rectangular} is computed by finding the closest structured matrix having the same kernel as the truncated SVD approximation. This is achieved by solving the inner minimization problem~\cite[eq. (f(R))]{slra-software}.
The computed trajectories (for one run) are presented in Figure~\ref{fig:sysid_Fro_rectangular}. 
\begin{figure}[h!]%
\centering
	\psfrag{x}{$t$}
	\psfrag{y}{\hspace*{-1.5cm}$y,\, y_{0},\, \hat y_{A1},\, \hat y_{slra},\, \hat y_{C}$}	\includegraphics[width=0.7\columnwidth]{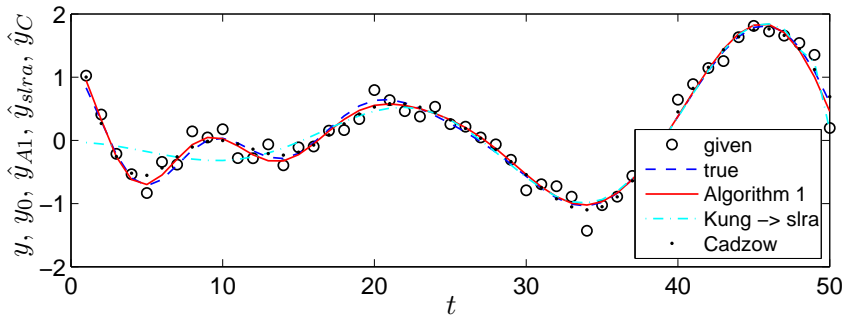}
\caption{Noisy data $y$, true data $y_0$ and the trajectories obtained from Algorithm~\ref{alg:reg_slra} ($\hat y_{A1}$), {\tt slra} ($\hat y_{slra}$), and Cadzow ($\hat y_{C}$), for the example \eqref{eq:sysid_y0}--\eqref{eq:sysid_y} with $m=5$. The computations are with respect to the Frobenius norm.}%
\label{fig:sysid_Fro_rectangular}%
\end{figure}

In a second experiment, 
we considered the almost square $25\times 26$ matrix ${\cal H}_{25}$ and ran the algorithms again.
The initial approximation was again obtained using the truncated SVD. 
Note that {\tt slra} can only be run on ${\cal H}_{5}$ but can profit from the initial approximation from ${\cal H}_{25}$, by running {\tt slra} after Kung's method.
The result obtained by {\tt slra} when initialized with Kung's method and the result obtained by Algorithm~\ref{alg:reg_slra} were similar to each other. Cadzow's algorithm also performed well in this experiment.
The numerical errors are presented in Table~\ref{tab:sysid_Fro_comparison_square}.
The computed trajectories (for one run) are presented in Figure~\ref{fig:sysid_Fro_square}. 
\begin{table}[htb]%
\centering
\caption{Numerical errors of Cadzow's algorithm, Kung's heuristic algorithm, {\tt slra} initialized with Kung's algorithm's solution, and the proposed algorithm (Algorithm~\ref{alg:reg_slra}), for the example \eqref{eq:sysid_y0}--\eqref{eq:sysid_y} with $m=25$.}
\begin{tabular}{l|r|r|l}
	 & $\|\Scal(y) - \Scal(\hat y)\|^2_F$ & $\|\Scal(y_0) - \Scal(\hat y)\|^2_F$ & Remarks\\
	\hline
	Cadzow & $12.5219$ & $ 1.4715$ &\\
	\hline
	{\tt Kung}~\cite{kung78} & $19.7567$ & $ 10.2969$ & heuristic\\
	\hline
	{\tt Kung} $\rightarrow$ {slra} & $\mathbf{12.5038}$ & $1.3495$ & \\
	\hline
	Algorithm~\ref{alg:reg_slra} & $12.5041$ & $\mathbf{1.3470}$ & $\frac{\|PL - {\cal P}_{\Scal}({PL})\|^2_F}{\|PL\|^2_F} < 10^{-22}$
\end{tabular}
\label{tab:sysid_Fro_comparison_square}
\end{table}
\begin{figure}[h!]%
\centering
	\psfrag{x}{$t$}
	\psfrag{y}{\hspace*{-1.5cm}$y,\, y_{0},\, \hat y_{A1},\, \hat y_{slra},\,\hat y_{C}$}	\includegraphics[width=0.7\columnwidth]{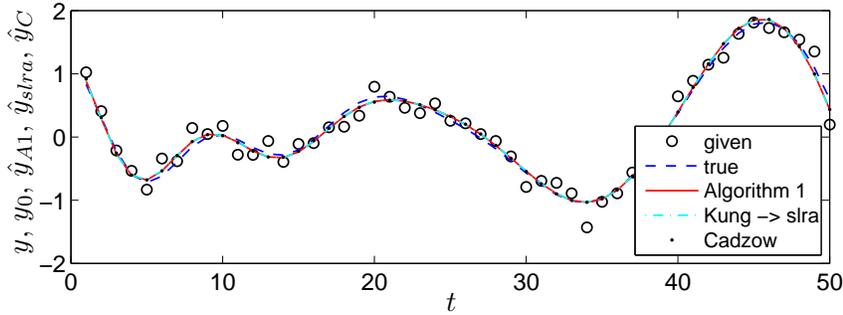}
\caption{Noisy data $y$, true data $y_0$ and the trajectories obtained from Algorithm~\ref{alg:reg_slra} ($\hat y_{A1}$), {\tt slra} ($\hat y_{slra}$), and Cadzow ($\hat y_{C}$), for the example \eqref{eq:sysid_y0}--\eqref{eq:sysid_y} with $m=25$. The computations are with respect to the Frobenius norm.}%
\label{fig:sysid_Fro_square}%
\end{figure} 

Different realizations of the example lead to slightly different numerical results. However, the main conclusions of this example stay the same.

\subsubsection*{Example: system identification with missing data}
In this example, we illustrate the fact that the proposed algorithm can work with matrices having {\it missing} elements. 
We continue with the above example but since Cadzow's algorithm cannot be applied to the missing data case directly, this time the objective function is with respect to the more standard parameter norm $\|p-\hat p\|^2_2$. 

Since the rank of the true Hankel matrix is $4$, standard algorithms need at least $5$ ($5=4+1$) consecutive data points for identification. However, in the example below, we removed every $5$th data point, so these algorithms cannot be applied directly. Algorithm~\ref{alg:reg_slra} and {\tt slra}, however, can be used.
The initial approximation for the missing values was provided by averaging their two neighboring data points, after which the initial approximation for the algorithms was obtained by the truncated SVD. 

We ran again two experiments with $m=5$ and $m=25$, respectively.
The numerical errors are presented in Table~\ref{tab:sysid_misssing_comparison_rectangular} and
Table~\ref{tab:sysid_misssing_comparison_square}, respectively.
The obtained trajectories from Algorithm~\ref{alg:reg_slra} and {\tt slra} (initialized with the solution of Kung's algorithm) are given in Figure~\ref{fig:sysid_missing_rectangular} and Figure~\ref{fig:sysid_missing_square}, respectively.
\begin{table}[h!]%
\centering
\caption{Numerical errors of the initial approximation (by SVD), {\tt slra}, Kung's heuristic algorithm, {\tt slra} initialized with Kung's algorithm's solution, and the proposed algorithm (Algorithm~\ref{alg:reg_slra}), for the example \eqref{eq:sysid_y0}--\eqref{eq:sysid_y} with $m=5$ and missing data. $\overline W$ is a diagonal matrix with diagonal consisting of zeros at the positions corresponding to the missing elements and ones otherwise.}
\hspace*{-1mm}\begin{tabular}{l|r|r|r|l}
	 & $\|y - \hat y\|^2_{\overline W}$ & $\|y_0 - \hat y\|^2_{I-\overline W}$ & $\|y_0 - \hat y\|_2^2$ & Remarks\\
	\hline
	{\tt init.approx.} & $8.4295$ & $1.9481$ & $10.2588$ \\
	\hline
	{\tt slra} & $ 8.4123$  & $1.9447$ & $10.2386$\\
	\hline
	{\tt Kung}~\cite{kung78} & $5.5200\cdot\! 10^{5}$ & $9.0292\cdot\! 10^4$ & $ 6.4178\cdot\! 10^5$ & heuristic\\
	\hline
	{\tt Kung} $\rightarrow$ {\tt slra} & $8.3232$ &  $ 2.3540$ &  $10.3141$ & \\
	\hline
	Algorithm~\ref{alg:reg_slra} &  ${\bf 0.9027}$ & ${\bf 0.0512}$ &  ${\bf 0.1551}$ & $\!\frac{\|PL - {\cal P}_{\Scal}({PL})\|^2_F}{\|PL\|^2_F} < \!10^{-25}$
\end{tabular}
\label{tab:sysid_misssing_comparison_rectangular}
\end{table}
\begin{figure}[h!]%
\centering
	\psfrag{x}{$t$}
	\psfrag{y}{\hspace*{-1.1cm}$y,\, y_{0},\, \hat y_{A1},\, \hat y_{slra}$}
	\includegraphics[width=0.7\columnwidth]{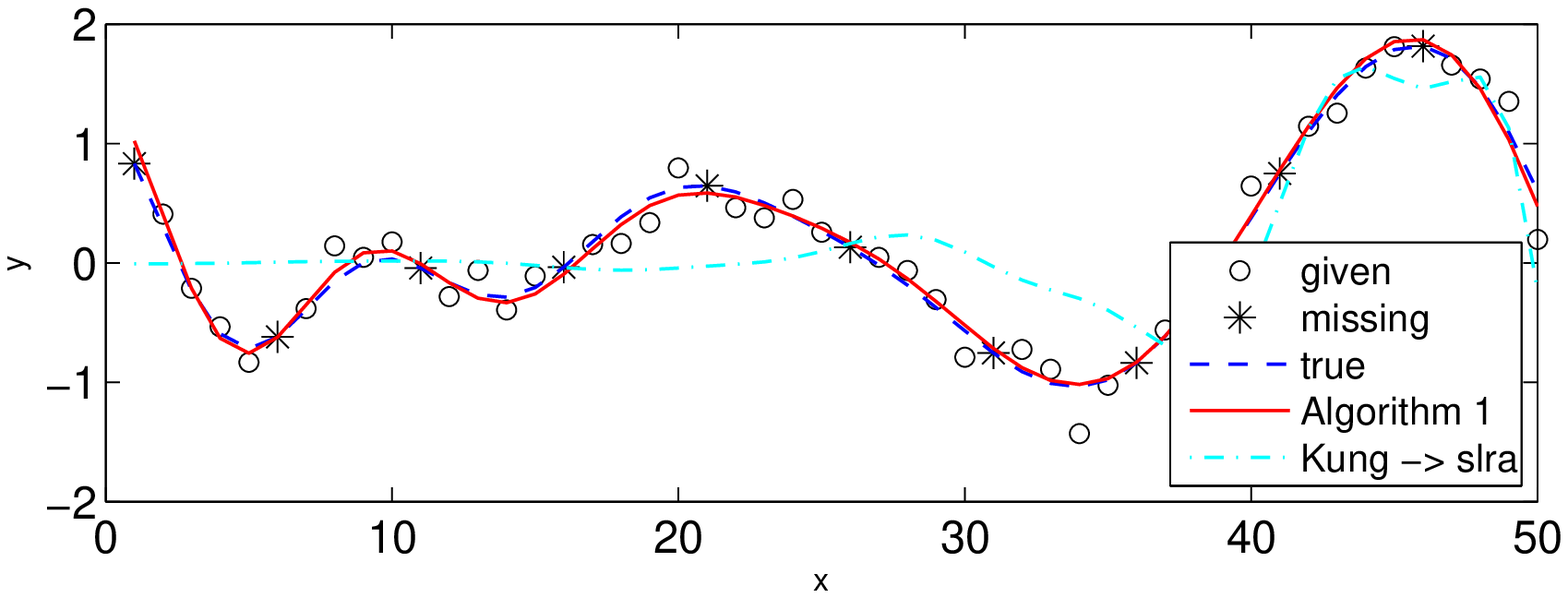}
\caption{Noisy data (subset of $y$), missing data (subset of $y_0$), true data $y_0$ and the trajectories obtained from Algorithm~\ref{alg:reg_slra} ($\hat y_{A1}$) and {\tt slra} ($\hat y_{slra}$) for the example \eqref{eq:sysid_y0}--\eqref{eq:sysid_y} with $m=5$ and missing data.}%
\label{fig:sysid_missing_rectangular}%
\end{figure}
\begin{table}[h!]%
\centering
\caption{Numerical errors of Kung's heuristic algorithm, {\tt slra} initialized with Kung's algorithm's solution, and the proposed algorithm (Algorithm~\ref{alg:reg_slra}), for the example \eqref{eq:sysid_y0}--\eqref{eq:sysid_y} with $m=25$ and missing data. $\overline W$ is a diagonal matrix with diagonal consisting of zeros at the positions corresponding to the missing elements and ones otherwise.}
\begin{tabular}{l|r|r|r|l}
	 & $\|y - \hat y\|^2_{\overline W}$ & $\|y_0 - \hat y\|^2_{I-\overline W}$ & $\|y_0 - \hat y\|_2^2$ & Remarks\\
	\hline
	{\tt Kung}~\cite{kung78} & $1.2480$ & $ 0.2722$ & $0.6935$ & heuristic\\
	\hline
	{\tt Kung} $\rightarrow$ {\tt slra} & $0.9228$ &  $0.1404$ &  $0.2750$ & \\
	\hline
	Algorithm~\ref{alg:reg_slra} &  ${\bf 0.9033}$ & ${\bf 0.0515}$ &  ${\bf 0.1429}$ & $\frac{\|PL - {\cal P}_{\Scal}({PL})\|^2_F}{\|PL\|^2_F} < 10^{-23}$
\end{tabular}
\label{tab:sysid_misssing_comparison_square}
\end{table}
\begin{figure}[h!]%
\centering
	\psfrag{x}{$t$}
	\psfrag{y}{\hspace*{-1.1cm}$y,\, y_{0},\, \hat y_{A1},\, \hat y_{slra}$}
	\includegraphics[width=0.7\columnwidth]{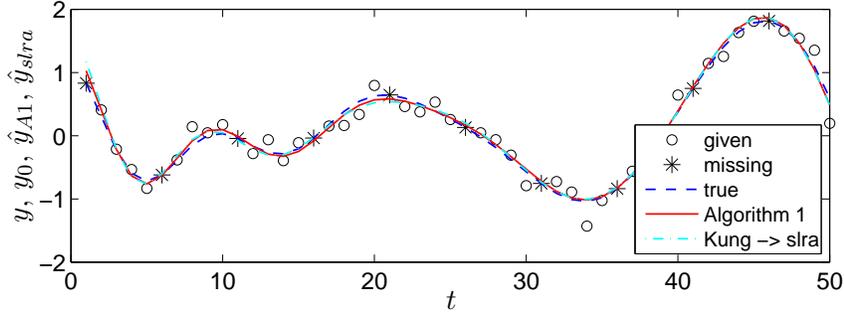}
\caption{Noisy data (subset of $y$), missing data (subset of $y_0$), true data $y_0$ and the trajectories obtained from Algorithm~\ref{alg:reg_slra} ($\hat y_{A1}$) and {\tt slra} ($\hat y_{slra}$) for the example \eqref{eq:sysid_y0}--\eqref{eq:sysid_y} with $m=25$ and missing data.}%
\label{fig:sysid_missing_square}%
\end{figure}
For $m=5$, Algorithm~\ref{alg:reg_slra} had the smallest approximation error. For $m=25$, the results of Algorithm~\ref{alg:reg_slra} and {\tt slra} (initialized with Kung's method) were similar, although Algorithm~\ref{alg:reg_slra} was still better.

{Different realizations of the example lead to slightly different numerical results. We also observed that for smaller noise variance, 
{\tt slra} and Algorithm~\ref{alg:reg_slra} compute the same solution, but for higher values of the noise 
Algorithm~\ref{alg:reg_slra} is generally more robust.}

\subsection{Approximate common divisor}
\label{ex:gcd}

Another application of structured low-rank approximation and thus of Algorithm~\ref{alg:reg_slra} is finding approximate common divisors of a set of polynomials. Existence of a nontrivial common divisor is a nongeneric property. Given a noisy observation of the polynomials' coefficients (or due to round-off errors in storing the exact polynomials coefficients in a finite precision arithmetic), the polynomials have a nontrivial common divisor with probability zero. Assuming that the noise free polynomials have a common divisor of a known degree, our aim in the approximate common divisor problem is to estimate the common divisor from the noisy data. The problem can be formulated and solved as Sylvester structured low-rank approximation problems, see \cite{slra-agcd}. Since the Sylvester matrix has fixed zero elements, in this example, we use the feature of Algorithm~\ref{alg:reg_slra} to work with structured matrices having {\em fixed elements}.

\subsubsection*{Background}

For a polynomial
\begin{equation*}
	a(z) \,\,=\,\, a_0 + a_1\, z + \cdots + a_{n}\, z^{n},
\end{equation*}
define the multiplication matrix
\begin{equation*}
	S_k(a)=\left.\begin{bmatrix}
    a_0 & a_1 & \cdots & a_{n} & & \mathbf{0}\\
    & \ddots & \ddots & &\ddots & \\
    \mathbf{0} & & a_0 & a_1 & \cdots & a_{n}
    \end{bmatrix}\right\}k\quad .
\end{equation*}
We consider three polynomials $a$, $b$, and $c$ and for simplicity let they be of the same degree~$n$. A basic result in computer algebra, see, e.g., \cite{KL98}, is that $a$, $b$, and $c$ have a nontrivial common divisor if and only if the generalized Sylvester matrix
\begin{equation}
\begin{bmatrix}
	S_{n}(b) & S_{n}(c)\\[1mm]
	S_{n}(a) & \mathbf{0}\\[1mm]
	\mathbf{0} & S_{n}(a)
\end{bmatrix}
\label{eq:Sylvester3a}
\end{equation}
has rank at most $3n-1$. Alternatively, one can consider another type of generalized Sylvester matrix~\cite{gen-sylvester2}
\begin{equation}
\begin{bmatrix}
	S_{n}(a) \\[1mm] S_{n}(b) \\[1mm] S_{n}(c)
\end{bmatrix},
\label{eq:Sylvester4}
\end{equation}
whose rank deficiency is equal to the degree of the greatest common divisor of $a$, $b$, and $c$. 
Formulation (\ref{eq:Sylvester4}) is more compact than the one in (\ref{eq:Sylvester3a}), especially if the number of polynomials is large.
These results are generalizable for arbitrary number of polynomials ($2, 3, \ldots$) of possibly different degrees and arbitrary order of the required common divisor.

In the case of inexact coefficients, the generalized Sylvester matrices are generically of full rank. The problem of finding the approximate common divisor is then transformed to the problem of approximating the matrix in (\ref{eq:Sylvester3a}) or in (\ref{eq:Sylvester4}) by low-rank matrices with the same structure. This can be done with Algorithm~\ref{alg:reg_slra} and with the alternative method {\tt slra}. For simplicity, in our example we take three polynomials of degree $2$ and desired (greatest) common divisor of order one (one common root). 

\subsubsection*{Example}
Let
\begin{equation*}
\begin{array}{ccccl}
	a(z) & = & 5 - 6z + z^2 & = & (1 - z)\,(5 - z),\\[1mm]
	b(z) & = & 10.8 - 7.4z + z^2 & = & (2 - z)\,(5.4 - z),\\[1mm]
	c(z) & = & 15.6 - 8.2z   +  z^2 & = & (3 - z)\,(5.2 - z).
\end{array}
\end{equation*}
Aiming at a common divisor of degree one, we approximate (\ref{eq:Sylvester3a}) and (\ref{eq:Sylvester4}) with, respectively, rank-$5$ and rank-$3$ matrices. The obtained solution with Algorithm~\ref{alg:reg_slra}, applied on \eqref{eq:Sylvester4} is
\begin{equation*}
\begin{array}{ccccc}
	\hat a(z) & = & 4.9991 -6.0046\,z + 0.9764\,z^2
	    & = & 0.9764\,\,(0.9928 - z)\,\textcolor{blue}{(5.1572 - z)},\\[1mm]
	\hat b(z) & = & 10.8010  -7.3946\,z + 1.0277\,z^2
	    & = & 1.0277\,\,(2.0378 - z)\,\textcolor{blue}{(5.1572 - z)},\\[1mm]
	\hat c(z) & = & 15.6001  -8.1994\,z +  1.0033\,z^2
	    & = & 1.0033\,\,(3.0149 - z)\,\textcolor{blue}{(5.1572 - z)},
\end{array}
\end{equation*}
with a common root $5.1572$. The approximation error was
$\|p-\hat p\|^ 2 = 0.0014,$ where $p$ is the vector of $9$ initial coefficients ($p=\begin{bmatrix}a_0 & a_1 & a_2 & \cdots & c_2\end{bmatrix}$) and 
$\hat p$ is the vector of $9$ coefficients of the approximations
($\hat p = \begin{bmatrix}\hat a_0 & \hat a_1 & \hat a_2 & \cdots & \hat c_2\end{bmatrix}$). 
	The roots of the original and approximating polynomials, as well as the polynomials themselves are plotted in Figure~\ref{fig:gcd_v2}.
\begin{figure}[htb]%
\subfigure[Locations of the roots.]{
    \psfrag{a}{${a}$}
    \psfrag{b}{${b}$}
    \psfrag{c}{${c}$}
    \psfrag{ah}{${\hat a}$}
    \psfrag{bh}{${\hat b}$}
    \psfrag{ch}{${\hat c}$}
    \psfrag{x}{\raisebox{-1mm}{roots}}
    \psfrag{y}{\hspace*{-1cm}$a\quad\quad\, b\quad\quad\, c$}
    \includegraphics[width=0.42\columnwidth]{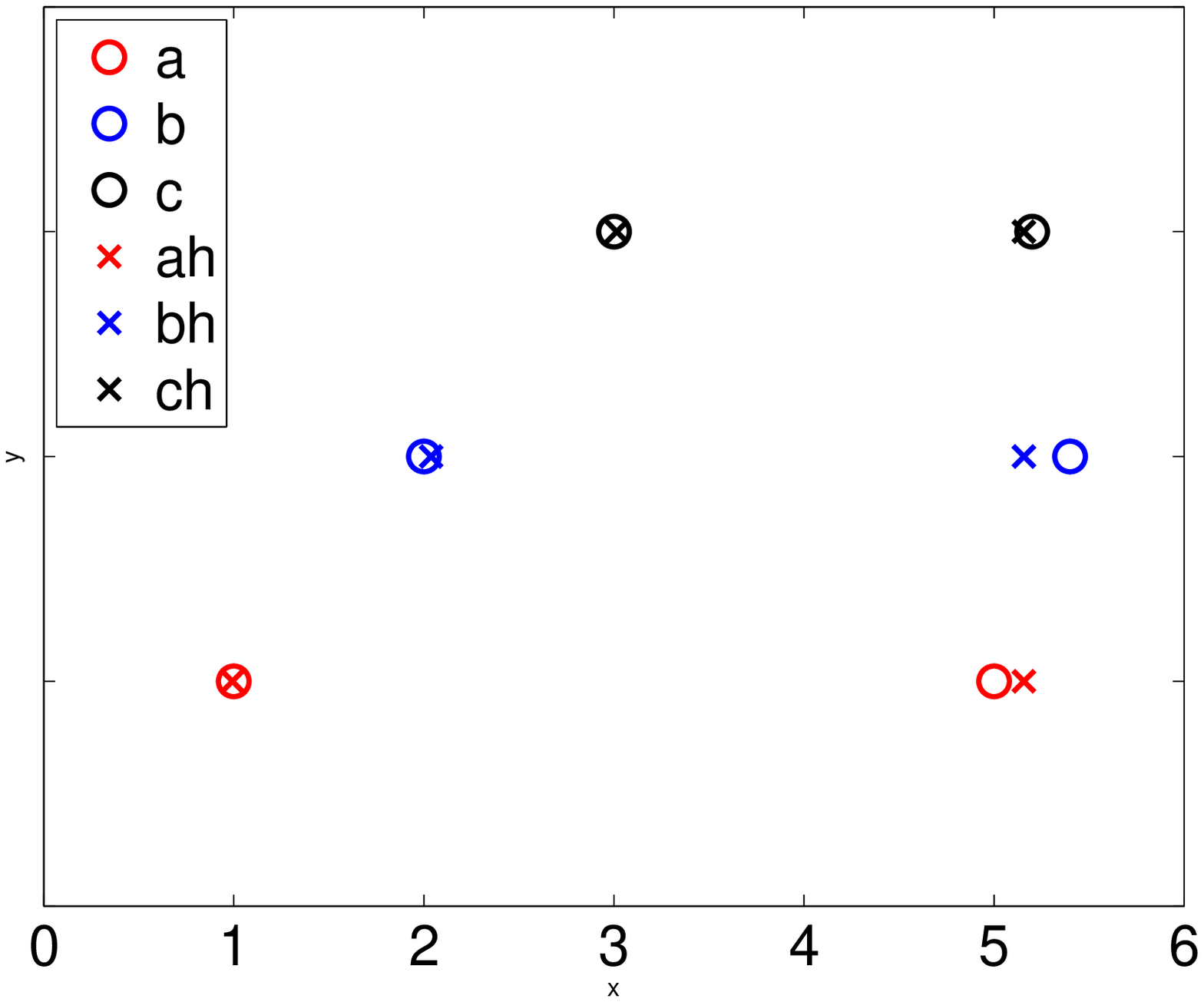}}
\subfigure[Polynomials as functions of $z.$]{
    \psfrag{a}{${a}$}
    \psfrag{b}{${b}$}
    \psfrag{c}{${c}$}
    \psfrag{ah}{${\hat a}$}
    \psfrag{bh}{${\hat b}$}
    \psfrag{ch}{${\hat c}$}
    \psfrag{x}{\raisebox{-1mm}{$z$}}
    \psfrag{y1}{\hspace*{-2mm}$a,\, b,\, c$}
    \psfrag{y2}{\hspace*{-2mm}$\hat a,\, \hat b,\, \hat c$}
    \includegraphics[width=0.52\columnwidth]{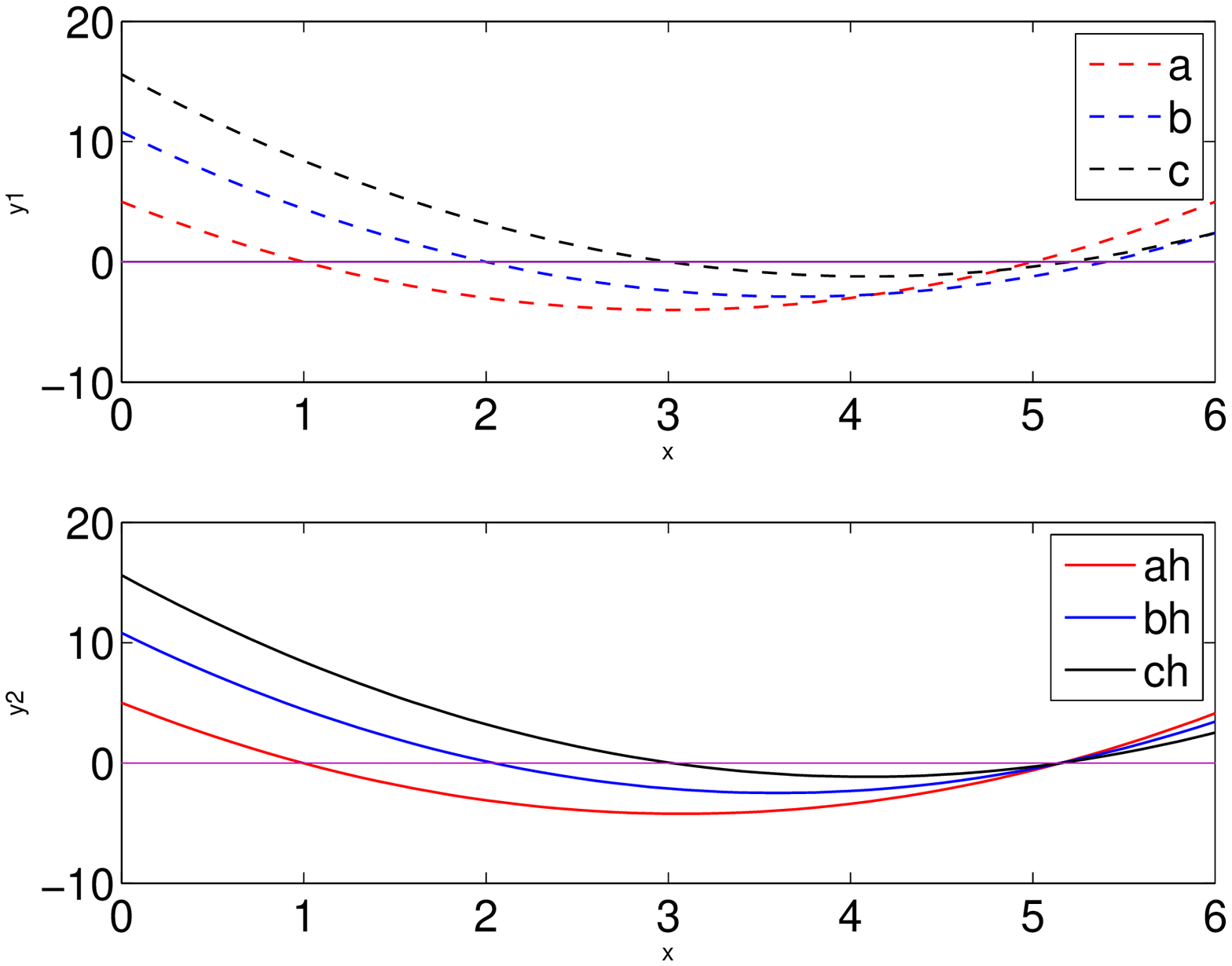}}
\caption{Results for the example of approximate common divisor of three polynomials.}%
\label{fig:gcd_v2}
\end{figure}
{The same results were obtained with {\tt slra}, applied on \eqref{eq:Sylvester3a}. 
Due to the non-convexity of the problem, Algorithm~\ref{alg:reg_slra} applied on \eqref{eq:Sylvester3a} computed a slightly different solution with comparable accuracy ($\|p-\hat p\|^ 2 = 0.0015$). Applying {\tt slra} on (\ref{eq:Sylvester4}) resulted in computing a common divisor of degree $2$, i.e., $\hat a= \hat b = \hat c$.}
Cadzow's algorithm performed well in both cases ((\ref{eq:Sylvester4}) and (\ref{eq:Sylvester3a})) and resulted in two other approximations with $\|p-\hat p\|^ 2 = 0.0014$ and $\|p-\hat p\|^ 2 = 0.0015$, respectively.

\subsection{Symmetric tensor approximation}
\label{sec:tensor}
Structured low-rank approximation can be applied to decompose and approximate complex symmetric tensors into a sum of symmetric rank-one terms.
With this application we also demonstrate how the proposed algorithm can deal with {\em missing} and {\em fixed} elements on quasi-Hankel structured matrices.

\subsubsection*{Background}
A complex tensor of dimension $n$ and order $d$, $\mathcal{T} \in \C^{\overbrace{n \times \cdots \times n}^d}$ is called symmetric if 
it is invariant under any permutation of the modes. A symmetric tensor $\mathcal{T}$ admits a symmetric decomposition of rank $r$ if it can be represented as a sum of $r$ rank-one symmetric terms:
\begin{equation}
\label{eq:symtdec}
	\mathcal{T} = \sum\limits_{k=1}^{r} \overbrace{\mathbf{v}_k \otimes \cdots \otimes \mathbf{v}_k}^{d},
\end{equation}
where $\mathbf{v}_k \in \C^{n}$. The problem of symmetric tensor decomposition is to find a (minimal) decomposition \eqref{eq:symtdec}.

In \cite{Brachat.etal10LAA-Symmetric}, it was shown that $\mathcal{T}$ has a decomposition of rank $r$ if and only if (excluding some nongeneric cases)
$$\rank {\Scal(\mathcal{T},h)} \le r,$$
where $\Scal(\mathcal{T},h)$ is a quasi-Hankel structured matrix constructed from the tensor
and the vector $h$ has unknown entries (latent variables). Therefore, the tensor decomposition problem is a low-rank matrix completion of the structured matrix $\Scal(\mathcal{T},h)$. Moreover, symmetric low-rank tensor approximation  is equivalent  to structured low-rank approximation with missing data. The main difficulty for this reformulation is that filling in missing data is nontrivial, and gives rise to multivariate polynomial systems of equations \cite{Brachat.etal10LAA-Symmetric}. 

Symmetric tensors are often represented as homogeneous polynomials \cite{Comon09conf-Tensors} using
$$
\mathcal{T} \longleftrightarrow \mathcal{T}(\mathbf{x}) := \mathcal{T} \times_{1} \mathbf{x} \times_{2} \mathbf{x} \cdots \times_{d} \mathbf{x},
$$
where $\mathbf{x} \in \C^{n}$ and `$\times_i$' is the tensor-vector product with respect to the $i$th mode of the tensor. With this representation, symmetric tensor decomposition is equivalent to the Waring problem \cite{Brachat.etal10LAA-Symmetric} of decomposing a homogeneous polynomial $\mathcal{T}(\mathbf{x})$ of degree $d$ into a sum of $d$-th powers of linear forms
$$
\mathcal{T}(\mathbf{x}) := 
\sum\limits_{k=1}^r (\mathbf{v}_k^\top \mathbf{x})^{d}.
$$ 
We will represent our results in the polynomial form.

\subsubsection*{Example}
Consider the example from \cite[\S 5.2]{Brachat.etal10LAA-Symmetric}, which is decomposition of a polynomial (a tensor of dimension $3$ and order $4$), with $\mathbf{x} = \begin{bmatrix} x_0 & x_1 & x_2 \end{bmatrix}^{\top}$,
$$
\mathcal{T}(\mathbf{x}) = 79x_0x_1^3+ 56x_0^2x_2^2 + 49x_1^2x_2^2 + 4x_0x_1x_2^2 + 57x_1x_0^3,
$$
into a sum of $6$ symmetric rank-one terms. (For dimension $3$ and degree $4$ complex symmetric tensors, the generic rank is $6$ \cite{Comon09conf-Tensors}). In this case, in order to compute the decomposition with the theory of \cite{Brachat.etal10LAA-Symmetric} it is crucial to fill in the missing data.

We considered the $10\times 10$ submatrix of the matrix in \cite[p.14]{Brachat.etal10LAA-Symmetric}. We fixed the non-missing data (by setting them in $S_0$), and computed the missing elements with Algorithm~\ref{alg:reg_slra}. The error on the deviation from the structure was around machine precision ($4.5\cdot 10^{-31}$).
The computed tensor decomposition was
\begin{equation}
\begin{array}{lrl}
	\mathcal{T}(\mathbf{x}) \approx & 
		6.94\, {\left(x_0 + 0.895\, x_1 + 0.604\, x_2\right)}^4 & \\[1mm]
	& + 6.94\, {\left(x_0 + 0.895\, x_1 - 0.604\, x_2\right)}^4 &\\[1mm]
	& - 4.94\, {\left(x_0 - 0.982\, x_1 + 0.657\, {i}\, x_2\right)}^4 &\\[1mm]
	& + 4.94\, {\left(x_0 - 0.982\, x_1  - 0.657\, {i}\, x_2\right)}^4 &\\[1mm]
	& -\left(1.99-11.9\, {i}\right){\left(x_0 + \left(0.128 + 0.308\, {i}\right)x_1\right)}^4 &\\[1mm] 
	& -\left(1.99 + 11.9\, {i}\right){\left(x_0 + \left(0.128 - 0.308\, {i}\right)x_1\right)}^4 &\hspace*{-3mm},
\end{array}
\label{eq:tensor_example}
\end{equation}
(where we have removed coefficients smaller than $10^{-12}$).
This is a different expansion from the one reported in \cite[p.15]{Brachat.etal10LAA-Symmetric}.
This can be expected, because for generic ranks the tensor decompositions are usually nonunique \cite{Comon09conf-Tensors}. The approximation error of \eqref{eq:tensor_example} on the normalized polynomial coefficients is $1.7421 \cdot 10^{-13}$. 

\begin{remark}
	Instead of the method used in \cite[p.15]{Brachat.etal10LAA-Symmetric}, we computed the vectors $\mathbf{v}_k$ by joint diagonalization of matrices $\mathbb{M}_{x_1}$ and $\mathbb{M}_{x_2}$ of the quotient algebra (see \cite{Brachat.etal10LAA-Symmetric} for the definition of these matrices). For joint diagonalization we used the method \cite{Rouquette.Najim01Estimation} from signal processing, where  approximate joint eigenvectors are computed by taking the eigenvectors of a linear combination of $\mathbb{M}_{x_1}$ and $\mathbb{M}_{x_2}$. This was needed because in the case of  multiple eigenvalues of $\mathbb{M}_{x_1}$ (which is the case of our computed decomposition), the method in \cite[p.15]{Brachat.etal10LAA-Symmetric} does not work correctly.
\end{remark}
\begin{remark}
Note that in case of (structured) matrix completion with exact data, the given data are incorporated in the matrix $S_0$. Thus, since no elements are being approximated, $\overline W = \mathbf{0}$ and the first term in the objective function \eqref{def:new}
$$\min_{P,\,L} \underbrace{\|p- \Ss^\dagger \vec{PL}\|^2_{\overline W}}_{=0} 
	+\lambda\|PL - {\cal P}_{\Scal}({PL})\|^2_F$$
vanishes.
The parameter $\lambda$ would not affect the resulting optimization problem and can also be removed.
The problem is thus reduced to the following problem
$$\min_{P,\,L}\|PL - {\cal P}_{\Scal}({PL})\|^2_F$$
and can be solved faster (since no iterations over $\lambda$ are necessary).
\end{remark}

In this example, the data were exact and the goal was to compute exact rank-$6$ decomposition. However, Algorithm~\ref{alg:reg_slra} with the full objective function \eqref{def:new} can be used to solve the more general problem of tensor low-rank approximation as well.

\section{Conclusions}
\label{sec:conclusions}
In this paper, we introduced a novel approach for solving the structure-preserving low-rank approximation problem. We used the image representation to deal with the low-rank constraint, and a penalty technique, to impose the structure on the approximation.
The original problem has been reduced to solving a series of simple least squares problems with exact solutions. We have discussed the properties of the proposed  local optimization algorithm and ensured that it can solve the weighted problem and deal with the cases of missing or fixed elements.
The proposed algorithm was tested on a set of numerical examples from system identification, computer algebra and symmetric tensor decomposition and compared favorably to existing algorithms.

The penalized structured low-rank approximation algorithm proposed in this paper is an attractive alternative to the kernel approach: it is more robust to the initial approximation (Section~\ref{ex:sysid}), allows us to use a simpler Sylvester matrix \eqref{eq:Sylvester4} in the GCD setting, and can be used for symmetric tensor decompositions, where the alternative {\tt slra} method experiences difficulties. In contrast to algorithms based on the kernel representation, the proposed structured low-rank approximation is designed for the problems requiring low ranks (small $r$). It is also worth noting that there are no restrictions on the values of the rank $r$. An efficient implementation of the algorithm and more detailed analysis of its applications in case of missing data and GCD computation are a topic of future research. 

\appendix
\section*{Appendix}

\begin{proof}[Proof of Lemma~\ref{lem:projection}.]
The closest matrix in $\image{\Scal}$ to an unstructured matrix $X$
is the solution of the minimization problem
$$\argmin{\hat X\,\in\, \image{\Scal}} \|X- \hat X\|^2_F,$$
where $\|\cdot\|_F$ stands for the Frobenius norm.
Equivalently, we need to solve
$$\min_{p\in\R^{n_p}} \|X - \Scal(p)\|^2_F,$$
which can be written as
\begin{equation}
\min_{p\in\R^{n_p}} \|\vec{X} - \vec{S_0} - \Ss p\|^2_2.
\label{min:projection}
\end{equation}
Since, by assumption {\bf (A)}, $\Ss$ has full column rank and $\Ss^\top \vec{S_0} =\mathbf{0},$
(\ref{min:projection}) is a least squares problem with unique solution 
$$p^\ast = (\Ss^\top \,\Ss)^{-1}\Ss^\top \vec{X-S_0}=(\Ss^\top \,\Ss)^{-1}\Ss^\top \vec{X} = \Ss^\dagger \vec{X}.$$  Thus,
$${\cal P}_{\Scal}(X) = \Scal(\Ss^\dagger \vec{X}),$$
which completes the proof.\hfill$ $
\end{proof}

\begin{remark}[Interpretation of the orthogonal projection]
The effect of applying $\Ss^\dagger$ on a vectorized $m\times n$ matrix $X$ is producing a  structure parameter vector 
by averaging the elements of $X,$ corresponding to the same $S_k$.  
Indeed, the product $\Ss^\top\vec{X}$ results in a vector containing the sums of the elements corresponding to each~$S_k$.
By assumption {\bf(A)}, 
$\Ss^\top \,\Ss$ is a diagonal matrix, with elements on the diagonal equal to the number of nonzero elements in each $S_k$, i.e., 
$$\Ss^\top \,\Ss=\begin{bmatrix}\|S_1\|^2_F & & \mathbf{0}\\ & \ddots &\\ \mathbf{0} & & \|S_{n_p}\|^2_F\\ \end{bmatrix}
=\begin{bmatrix}\textnormal{nnz}(S_1) & &  \mathbf{0}\\ & \ddots &\\  \mathbf{0}& & \textnormal{nnz}(S_{n_p})\\ \end{bmatrix},$$
where \textnormal{nnz} stands for the number of nonzero elements.
Therefore multiplying by $(\Ss^\top \,\Ss)^{-1}\Ss^\top$ corresponds to averaging.

In particular, applying $\Ss^\dagger$ on a (vectorized) structured matrix extracts its structure parameter vector, since
$$\Ss^\dagger\, \vec{\Scal(p)} = (\Ss^\top \,\Ss)^{-1}\Ss^\top \,\Ss\,p = p.$$ 
\end{remark}

\vspace*{1mm}
\begin{proof}[Proof of Lemma~\ref{lem:LS_reformulation}.]
Using the following well-known equality
\begin{equation*}
\vec{XYZ}= (Z^\top\otimes X)\,\vec{Y},
\label{eq:vec}
\end{equation*}
we have
\begin{equation}
\vec{PL} = (I_n\otimes P)\,\vec{L} = (L^\top \otimes I_m)\,\vec{P}.
\label{eq:vec_PL}
\end{equation}

Consider first problem (\ref{def:subproblem1}). Problem (\ref{def:subproblem2}) can be solved in a similar way.
Using \eqref{eq:vec_proj} and \eqref{eq:vec_PL}, (\ref{def:subproblem1}) can be reformulated as
\begin{equation*}
\begin{array}{ll}
	& \displaystyle{\min_{L}} \|p- \Ss^\dagger\, \vec{PL}\|^2_{\overline W} +
		\lambda\|PL - {\cal P}_{\Scal}(PL)\|^2_F\\[3mm]
	\Longleftrightarrow\quad\mbox{ } & \displaystyle{\min_{L} \|\overline M(p-\Ss^\dagger\,\vec{PL})\|^2_2 +
		\lambda\|\vec{PL} - \vec{{\cal P}_{\Scal}({PL})}\|^2_2,}\\[3mm]
	\Longleftrightarrow & \displaystyle{\min_{L} \|\overline M (p-\Ss^\dagger\,\vec{PL})\|^2_2
		+\lambda\|\vec{PL} -  \vec{S_0} - \Pi_{\Ss}\,\vec{PL}\|^2_2,}\\[3mm]
	\Longleftrightarrow & \displaystyle{\min_{L} \|\overline M\Ss^\dagger\,\vec{PL} -\overline Mp\|^2_2 
		+\|\sqrt{\lambda}\Pi_{\Ss_\perp}\,\vec{PL} 
		-\sqrt{\lambda}\vec{S_0}\|^2_2,}\\[3mm]
	\Longleftrightarrow & \displaystyle{\min_{L} \left\|
		\begin{bmatrix}\overline M\,\Ss^\dagger\\[2mm]\sqrt{\lambda}\Pi_{\Ss_\perp}\end{bmatrix}\,\vec{PL}
		-\begin{bmatrix}\overline Mp\\[2mm]\sqrt{\lambda}\vec{S_0}\end{bmatrix}\right\|^2_2,}\\[6mm]		
	\Longleftrightarrow & \displaystyle{\min_{L} \left\|
			\begin{bmatrix} \overline M\, \Ss^\dagger\\[2mm]\sqrt{\lambda}\Pi_{\Ss_\perp}\end{bmatrix}\,(I_n\otimes P)\,\vec{L}
		-\begin{bmatrix} \overline M p\\[2mm]\sqrt{\lambda}\vec{S_0}\end{bmatrix}\right\|^2_2.}
\end{array}
\label{min2_}
\end{equation*}

The derivation for $P$ is analogous.\hfill
\end{proof}

\begin{proof}[Proof of Proposition~\ref{prop:feasibility}.]
We need to prove that
\[
\calJ^{\top}_c(P,L)\, c(P,L) = 0\quad  \Longrightarrow \quad c(P,L) = 0.
\]
Recall from \eqref{eq:constraint2} and \eqref{eq:constraint_gradients} that
\[
\begin{split}
c(P,L)
& = \Pi_{\bfS_\bot} \mvec(PL) 
= \Pi_{\bfS_\bot} (L^{\top} \otimes I_m) \mvec(P)
= \Pi_{\bfS_\bot} (I_n \otimes P) \mvec(L),\\[1mm]
\calJ_c(P,L) & = \Pi_{\bfS_\bot} \bmx L^{\top} \otimes I_m  & I_n \otimes P \emx.
\end{split}
\]
The expression $\calJ^{\top}_c(P,L)\, c(P,L) = 0$ is  then equivalent to the system of equations
\begin{equation}\label{eq:system}
\begin{cases}
&(L \otimes I_m)  \Pi_{\bfS_\bot} \Pi_{\bfS_\bot}  (L^{\top} \otimes I_m) \mvec(P) = 0,\\[1mm]
&(I_n \otimes P^{\top})\,  \Pi_{\bfS_\bot} \Pi_{\bfS_\bot} (I_n \otimes P) \mvec(L) = 0.
\end{cases}
\end{equation}
If we denote $A := \Pi_{\bfS_\bot} (L^{\top} \otimes I_m)$, $B := \Pi_{\bfS_\bot} (I_n\otimes P)$,
\eqref{eq:system} is equivalent to
\[\hspace*{-1mm}
\begin{cases}
&\!\!\!A^{\top}A\mvec(P) = 0,\\
&\!\!\!B^{\top}B\mvec(L) = 0.
\end{cases}
\iff 
\begin{cases}
&\!\!\!A\mvec(P) = 0,\\
&\!\!\!B\mvec(L) = 0.
\end{cases}
\iff
\Pi_{\bfS_\bot} \mvec(PL) = 0
\iff c(P,L) = 0,
\]
which completes the proof.\hfill
\end{proof}

\begin{proof}[Proof of Lemma~\ref{lem:rank_2nd_matrix}.]
	A vector $\mbox{vec}^\top(X)$, with $X\in\R^{m\times n}$, is in the left kernel of 
	$\begin{bmatrix}L^\top \otimes I_m & I_n\otimes P\end{bmatrix}$ when
	\[\begin{array}{rcl}
		\hspace*{-2mm}\mbox{vec}^\top\!(X) \bmx L^\top \otimes I_m & I_n\otimes P\emx = \mathbf{0}
		& \iff & 	\mbox{vec}^\top\!(X) (L^\top \otimes I_m) = \mathbf{0} \quad\mbox{and }\\[1mm]
		& 		& \mbox{vec}^\top\!(X)(I_n\otimes P) = \mathbf{0}\\[1mm]
		& \iff & XL^\top = \mathbf{0} \mbox{ and } P^\top X = \mathbf{0} \\[1mm]
		& \iff & X = P_{\bot} Y L_{\bot}, \mbox{ for some } Y \\[1mm]
		& \iff & \vec{X} = (L_{\bot}^{\top} \otimes P_{\bot}) \vec{Y}  \\[1mm]
		& \iff & X \in\mathop{\mathrm{image}} (L_{\bot}^{\top} \otimes P_{\bot}),
	\end{array}
	\]
	where 
	we have used the well-known equality $\vec{XYZ}= (Z^\top\otimes X)\,\vec{Y}$,\linebreak
	$P_{\bot} \in \bbR^{m\times (m-r_P)}$ and  $L_{\bot} \in \bbR^{(n-r_L) \times n}$ are orthogonal complements of $P$ and $L$, respectively, and
$\rank{P} = r_P$ and $\rank{L} = r_L.$ Then,
	\begin{equation*}
	\begin{array}{rcl}
	\rrank (\begin{bmatrix}L^\top \otimes I_m & I_n\otimes P\end{bmatrix}) 
		& = & mn -\dim (L_{\bot}^{\top} \otimes P_{\bot})\\[1mm]
		& = & mn - (m-r_P)(n-r_L)\\[1mm]
		& \leq & mr+nr-r^2.
	\end{array}
	\label{eq:rank_2nd_matrix}
	\end{equation*}
	The equality holds if and only if $r_P = r_L =r$.\hfill
\end{proof}

\begin{proof}[Proof of Lemma~\ref{prop:prop2}]
\quad
\begin{enumerate}
\item Since $\rrank (\mathcal{J}_{\widetilde{c}} (P_*,L_*)) = mr+nr-r^2$, it follows from Lemma~4.4 that the row span of $\mathcal{J}_{\widetilde{c}} (P_*,L_*)$ coincides with the row span of $\bmx L^{\top}_* \otimes I_{m} &  I_{n} \otimes P_* \emx$. We note that the cost function  $f(P,L) = \|D-PL\|^2_W$ can be expressed as $f(P,L) = g(\vec{PL})$, and thus its gradient  at the limit point can be expressed as 
\[
\nabla f(P_*,L_*) = \bmx L^{\top}_* \otimes I_{m} & I_{n} \otimes P_* \emx^{\top} \nabla g( \vec{P_*,L_*})
\]
and is also in the row span of $\bmx L^{\top}_* \otimes I_{m} &  I_{n} \otimes P_* \emx$. Thus the equation
\[
\mathcal{J}^{\top}_{\widetilde{c}} (P_*,L_*) \nu = \nabla f(P_*,L_*)
\]
is consistent and its set of solutions is an affine subspace of dimension
\[
(mn-n_p) - (mr+nr-r^2)  = (m-r)(n-r) -n_p.
\]
\item Since $\rrank (\bmx L^{\top}_* \otimes I_{m} &  I_{n} \otimes P_* \emx) = mr+nr-r^2$, from Lemma~4.4 we have that $\rrank (P_*) = \rrank (L_*) = r$.  Since $\mathcal{J}_{\widetilde{c}} (P,L)$ depends continuously on $(P,L)$, then
$\rrank \mathcal{J}_{\widetilde{c}} (P,L) = mr+nr-r^2$ in a neighborhood of $(P_*,L_*)$.
\end{enumerate}
\end{proof}


\end{document}